\documentclass{amsart}

\newtheorem{theorem}{Theorem}
\newtheorem{lemma}[theorem]{Lemma}
\newtheorem{proposition}[theorem]{Proposition}
\theoremstyle{definition}
\newtheorem{definition}[theorem]{Definition}
\theoremstyle{remark}
\newtheorem{remark}[theorem]{Remark}
\newtheorem{corollary}[theorem]{Corollary}
\numberwithin{theorem}{section}

\usepackage{amssymb}
\usepackage{empheq}
\usepackage{stmaryrd}
\usepackage{enumerate}
\usepackage{multicol}
\usepackage[shortlabels]{enumitem}
\usepackage{calc}
\usepackage{tikz}
\usepackage{listings}
\usepackage{csquotes}

\lstdefinestyle{derivation} {
    breakatwhitespace=false,         
    breaklines=true,                 
    captionpos=b,                    
    keepspaces=true,                 
    numbers=left,                    
    numbersep=5pt,
    escapechar=\&
}

\usetikzlibrary{automata,positioning}
\usetikzlibrary{arrows}

\usepackage[left=2.0cm,%
                right=2.0cm,%
                top=2.5cm,%
                bottom=2.5cm,%
                headheight=12pt,%
                a4paper]{geometry}%

\begin{document}

\title{A note on strong axiomatization of G\"odel Justification Logic}

\author{Nicholas Pischke}
\address{Hoch-Weiseler Str. 46, Butzbach, 35510, Hesse, Germany}
\email{pischkenicholas@gmail.com}


\begin{abstract}
Justification logics are special kinds of modal logics which provide a framework for reasoning about \emph{epistemic justifications}. For this, they extend classical boolean propositional logic by a family of necessity-style modal operators ``$t:$``, indexed over $t$ by a corresponding set of justification terms, which thus explicitly encode the \emph{justification} for the necessity assertion in the syntax. With these operators, one can therefore not only reason about modal effects on propositions but also about dynamics inside the justifications themselves. We replace this classical boolean base with G\"odel logic, one of the three most prominent fuzzy logics, i.e. special instances of many-valued logics, taking values in the unit interval $[0,1]$, which are intended to model inference under vagueness. We extend the canonical possible-world semantics for justification logic to this fuzzy realm by considering fuzzy accessibility- and evaluation-functions evaluated over the minimum t-norm and establish strong completeness theorems for various fuzzy analogies of prominent extensions for basic justification logic.
\end{abstract}

\maketitle

\section{Introduction}
Epistemology and its core notions like knowledge, belief, truth and justification have, since Hintikka's seminal work \cite{Hin1962}, found natural formalizations in the realm of modal logics and in their fundamental semantical interpretation over Kripke's possible-world models. The formal development of justification logic began with the so called \emph{logic of proofs} which extends basic propositional logic by a family of modal, necessity-style unary operators, introducing formulas of the form $t:\phi$, where the indexing parameter $t$ ranges over a corresponding set of \emph{proof terms}. This logic was devised by Artemov in \cite{Art1995}, \cite{Art2001}, to provide an arithmetic provability semantics for intuitionistic logic and bridge intuitionistic logic, the modal logic $\mathcal{S}4$ and formal arithmetic, a possibility anticipated by G\"odel in \cite{Goe1933}, \cite{Goe1938} where he provided an embedding of intuitionistic logic into $\mathcal{S}4$ and conceptualized the provability interpretation of the modality $\Box$ in $\mathcal{S}4$. However, an explicit embedding of $\mathcal{S}4$ into formal arithmetic was still left out. The missing link was then provided by the logic of proofs $\mathcal{LP}$, for which Artemov provided an embedding into formal arithmetic via the \emph{Arithmetic Completeness Theorem} and an embedding of $\mathcal{S}4$ into $\mathcal{LP}$ via \emph{Artemov's Realization Theorem}, assigning explicit \emph{proof terms} to necessity-statements, which forms an intricate relation between classical modal logics and justification logics.\\

From a modern perspective, $\mathcal{LP}$ is just one of various logical systems in the framework of justification logic, similarly to $\mathcal{S}4$'s position inside the common framework of classical modal logics. Kripke-style possible-world semantics for the logic of proofs was introduced in \cite{Fit2003}, \cite{Fit2005}, with the structures prominently called \emph{Fitting models}, and later naturally extended to the various other representatives of the class of justification logics. As later observed, the Realization Theorem also extends to other representatives of the respective frameworks of classical modal and justification logics. For a comprehensive overview over the framework of justification logics, see e.g. \cite{Art2008}.\\

G\"odel logic on the other hand is a very prominent example of a many-valued logic, with evaluations into the unit interval $[0,1]$, dating back to a paper of G\"odel \cite{Goe1932} where he introduced finite valued versions to provide an infinite family of logics between intuitionistic and classical logic with respect to expressive strength. The version taking values in the unit interval $[0,1]$ was first studied by Dummett in \cite{Dum1959} who also provided the first proof calculus in terms of a simple extension of a common Hilbert calculus for intuitionistic logic. A first-order variant of this infinitely-valued version was followingly studied by Horn \cite{Hor1969}(see also \cite{BPZ2007}). Besides of this intuitionistic access to G\"odel logic, a prominent different approach, and the one followed in this paper, is given via the route of mathematical fuzzy logic, deriving from the concept of fuzzy sets originating in Zadeh's landmark work \cite{Zad1965}, and originating mainly from the seminal monograph of H\'ajek \cite{Haj1998}. Semantically, fuzzy logics are defined over the notion of a t-norm (triangular norm, see e.g. \cite{KMP2000}) as a truth function for conjunction, where G\"odel logic results in the case of considering the minimum t-norm as a particular choice.\\

From an application perspective, it seems very natural to combine justification logics with a notion of vagueness (to model e.g. uncertain epistemic assertions) and thus in a more explicit manner to combine fuzzy logic and particularly G\"odel logic with justification logic in the sense of adding justification modalities to the underlying propositional language and evaluating them in a many-valued context. Similarly as fuzzy logic has proved itself to be one of the right tools to express and analyze vague propositional assertions, we believe that fuzzy justification logic shall be a right tool to model uncertain (explicit) epistemic assertions.

Examples for combinations of justification logic with other systems (or notions) of vagueness (or probability) include Milnikel's logic of uncertain justifications \cite{Mil2014} where Milnikel introduces a graded justification operator $t:_r\phi$ for $r\in\mathbb{Q}\cap [0,1]$ with the intended meaning of $r$ being the least degree of confidence in ``$t$ being a justification for $\phi$``; and recently also the development of probabilistic justification logic, see e.g. \cite{KOS2016}, and possibilistic justification logic, see e.g. \cite{FL2015}.\\

At the time, the only literature on \emph{fuzzy} justification logics is Ghari's work in \cite{Gha2014} and \cite{Gha2016} where he introduced the notion of fuzzy Fitting models (with a crisp accessibility function) for respective t-norms, here later redefined for the case of G\"odel logic, and especially investigated the extension of Pavelka-style fuzzy logic. However, the study of models with graded accessibility functions and the derivation of strong completeness theorems was still left open in any case.\\

In this note, we concretize and expand the work of Ghari in \cite{Gha2014} and thus continue to expand the realm of fuzzy justification logics. We focus on G\"odel logics as an underlying base logic and introduce respective fuzzy Fitting models with \emph{fuzzy} accessibility functions. The main part is occupied with the study of fuzzy analogies of the most prominent justification logics and their semantics as well as their axiomatizations where we establish strong completeness results in every case. To approach this, we translate formulas containing modalities into an augmented basic propositional language and use the strong standard completeness of the underlying propositional systems. In general, we rely on various concepts of standard G\"odel modal logic, i.e. propositional G\"odel logic expanded by a classical necessity and possibility modality $\Box$ and $\Diamond$, originating in \cite{CR2009},\cite{CR2010} and \cite{CR2015}, (studying the necessity and possibility fragments as well as the Bi-modal variant respectively), where especially the structure of our approach to proving strong standard completeness is derived from. For a comprehensive overview over modal fuzzy logic and related advancements to completeness results for various modal fuzzy logics over models with a crisp accessibility function, see also e.g. \cite{Vid2015}. However, we do not recap the main notions from these works as they will be introduced in their adapted form for the case of fuzzy (G\"odel) justification logic during the note as appropriate. In the end, we give some starting points for directions of future work.
\section{Preliminaries}
The basis for our further investigations is propositional $[0,1]$-valued G\"odel logic. Formally, for this we fix a standard propositional language
\[
\mathcal{L}_0(X):\phi::=\bot\mid p\mid(\phi\land\phi)\mid(\phi\rightarrow\phi)
\]
with $p\in X$ for a countably infinite set of variables $X$. We may fix a \emph{standard} set of propositional variables with $Var:=\{p_i\mid i\in\mathbb{N}\}$. As an abbreviation, we then also write $\mathcal{L}_0:=\mathcal{L}_0(Var)$. As usual in propositional logics, we omit the outermost brackets if convenient. Other common connectives are introduced as abbreviations, i.e. we set
\begin{enumerate}
\item $\neg\phi:=(\phi\rightarrow\bot)$,
\item $\phi\leftrightarrow\psi:=(\phi\rightarrow\psi)\land(\psi\rightarrow\phi)$,
\item $\phi\lor\psi:=((\phi\rightarrow\psi)\rightarrow\psi)\land ((\psi\rightarrow\phi)\rightarrow\phi)$,
\item $\top:=(\bot\rightarrow\bot)$.
\end{enumerate}
Semantics for propositional G\"odel logics is defined via truth value assignments in the unit interval $[0,1]$, where conjunction is evaluated over the minimum t-norm $\min\{x,y\}$ for $x,y\in [0,1]$, in the following denoted with $x\odot y$, and implication over its residuum $\Rightarrow$, that is the \emph{unique} function $\Rightarrow:[0,1]\times [0,1]\to [0,1]$ satisfying
\[
x\odot y\leq z\text{ iff }x\leq y\Rightarrow z.
\]
In the case of the minimum t-norm, the residuum has the following description:
\[
x\Rightarrow y=\begin{cases}y,&\text{ if }x>y\\1,&\text{ otherwise}\end{cases}.
\]
\begin{definition}
A \emph{propositional assignment} is a function $e:X\to [0,1]$. This function can be naturally extended to a propositional G\"odel evaluation over $\mathcal{L}_0(X)$ by the following recursive definitions:
\begin{itemize}
\item $e(\bot)=0$,
\item $e(\phi\land\psi)=e(\phi)\odot e(\psi)$,
\item $e(\phi\rightarrow\psi)=e(\phi)\Rightarrow e(\psi)$.
\end{itemize}
\end{definition}
An evaluation $e$ may also be extended to sets of formulas $\Gamma$ via
\[
e(\Gamma):=\inf\{e(\phi)\mid\phi\in\Gamma\},
\]
where the infimum of the empty set is defined to be $1$. For the derived connectives, simplifying the corresponding evaluations yield the following expressions:
\begin{enumerate}
\item $e(\neg\phi)=\begin{cases}1,&\text{ if }e(\phi)=0\\0,&\text{ otherwise}\end{cases}$,
\item $e(\phi\leftrightarrow\psi)=\begin{cases}1,&\text{ if }e(\phi)=e(\psi)\\e(\phi)\odot e(\psi),&\text{ otherwise}\end{cases}$,
\item $e(\phi\lor\psi)=\max\{e(\phi),e(\psi)\}$,
\item $e(\top)=1$.
\end{enumerate}
The resulting truth functions corresponding to $\neg,\leftrightarrow$ and $\lor$ are respectively denoted by $\sim,\Leftrightarrow$ and $\oplus$. With $\mathsf{Ev}(\mathcal{L}_0(X))$, we denote the set of all evaluations of $\mathcal{L}_0(X)$ into $[0,1]$, that is the set of all maps $e:X\to [0,1]$, extended to $\mathcal{L}_0(X)$ by the above definition.\\

Before proceeding with other semantic notions, we want to note some properties of the minimum t-norm $\odot$ and its derived functions.
\begin{lemma}\label{lem:mintnormprop}
Let $x,y,z,x',y'\in [0,1]$:
\begin{enumerate}
\item If $x\leq x'$, $y\leq y'$, then $x\odot y\leq x'\odot y'$.
\item If $y\leq y'$, then $x\Rightarrow y\leq x\Rightarrow y'$.
\item If $x\geq x'$, then $x\Rightarrow y\leq x'\Rightarrow y$.
\item If $x\geq x'$, then $\sim x\leq\sim x'$.
\end{enumerate}
\end{lemma}
The proof is very elementary and thus omitted here.\\

From these definitions regarding semantic evaluations, analogues for the case of G\"odel logic of other common semantical notions can now be derived.
\begin{definition}
Let $\Gamma\cup\{\phi\}\subseteq\mathcal{L}_0(X)$. Then we say that
\begin{enumerate}
\item $\Gamma$ \emph{entails} $\phi$, $\Gamma\models_{\leq}\phi$, iff $\forall e\in\mathsf{Ev}(\mathcal{L}_0(X)):e(\Gamma)\leq e(\phi)$,
\item $\Gamma$ \emph{1-entails} $\phi$, $\Gamma\models\phi$, iff $\forall e\in\mathsf{Ev}(\mathcal{L}_0(X)):e(\psi)=1$ for all $\psi\in\Gamma$ implies $e(\phi)=1$.
\end{enumerate}
\end{definition}
However, as observed by Baaz and Zach, these two notions of semantic inference coincide (for countable sets of premises).
\begin{lemma}[Baaz, Zach \cite{BZ1998}]
For any $\Gamma\cup\{\phi\}\subseteq\mathcal{L}_0(X)$: $\Gamma\models_\leq\phi$ iff $\Gamma\models\phi$.
\end{lemma}
Through the approach to G\"odel logics via the framework of fuzzy logics, we consider an extension of H\'ajek's proof calculus for basic fuzzy logic $\mathbf{BL}$ \cite{Haj1998} by the idempotency axiom for conjunction as the corresponding proof calculus for axiomatizing the above defined semantic consequence relation of basic propositional G\"odel logic.
\begin{definition}
Let $\mathcal{G}$ be the Hilbert-style calculus given by the following axiom schemes\footnote{The numbering follows H\'ajek's original presentation in \cite{Haj1998}.} and rules:
\begin{description}
\item [($A1$)] $(\phi\rightarrow\psi)\rightarrow((\psi\rightarrow\chi)\rightarrow(\phi\rightarrow\chi))$
\item [($A2$)] $(\phi\land\psi)\rightarrow\phi$
\item [($A3$)] $(\phi\land\psi)\rightarrow(\psi\land\phi)$
\item [($A5a$)] $(\phi\rightarrow(\psi\rightarrow\chi))\rightarrow((\phi\land\psi)\rightarrow\chi)$
\item [($A5b$)] $((\phi\land\psi)\rightarrow\chi)\rightarrow(\phi\rightarrow(\psi\rightarrow\chi))$
\item [($A6$)] $((\phi\rightarrow\psi)\rightarrow\chi)\rightarrow(((\psi\rightarrow\phi)\rightarrow\chi)\rightarrow\chi)$
\item [($A7$)] $\bot\rightarrow\phi$
\item [($G4$)] $\phi\rightarrow(\phi\land\phi)$
\item [($MP$)] From $\phi$ and $\phi\rightarrow\psi$, infer $\psi$.
\end{description}
We denote a deduction of $\phi\in\mathcal{L}_0(X)$ in $\mathcal{G}$, or from a set of premises $\Gamma\subseteq\nobreak\mathcal{L}_0(X)$, by $\vdash_\mathcal{G}\phi$ and $\Gamma\vdash_\mathcal{G}\phi$ respectively.
\end{definition}
\begin{lemma}[H\'ajek \cite{Haj1998}]\label{lem:goedellogicthms}
$\mathcal{G}$ proves the following formulas:
\begin{enumerate}
\item $\phi\rightarrow (\psi\rightarrow\phi)$,
\item $\phi\rightarrow\phi$,
\item $\phi\rightarrow(\psi\rightarrow\chi)\rightarrow ((\phi\rightarrow\psi)\rightarrow(\phi\rightarrow\chi))$.
\end{enumerate}
\end{lemma}
While item (1) and (2) are even theorems for H\'ajek's basic logic $\mathbf{BL}$, item (3) is a particular feature of G\"odel logic, distinguishing it from the other prominent t-norm based logics. This lemma is also the reason why the usual proof of the classical deduction theorem works in G\"odel logic.
\begin{theorem}[Strong Standard Completeness, H\'ajek \cite{Haj1998}]\label{thm:gssc}
For any\\ $\Gamma\cup\{\phi\}\subseteq\mathcal{L}_0(X)$:
\[
\Gamma\vdash_\mathcal{G}\phi\text{ iff }\Gamma\models\phi.
\]
\end{theorem}
\section{G\"odel-Fitting models}
\begin{definition}
The \emph{language of G\"odel justification logic} $\mathcal{L}_J$ is defined by the BNF
\[
\mathcal{L}_J:\phi::= \bot\mid p\mid(\phi\land\phi)\mid(\phi\rightarrow\phi)\mid t:\phi
\]
with $p\in Var:=\{p_i\mid i\in\mathbb{N}\}$ as before and $t\in Jt$ where
\[
Jt:t::=x\mid c\mid [t\cdot t]\mid [t+t]\mid\; !t\mid\; ?t
\]
with $x\in V:=\{x_i\mid i\in\mathbb{N}\}$ \emph{variable symbols} and $c\in C:=\{c_i\mid i\in\mathbb{N}\}$ \emph{constant symbols}.
\end{definition}
The same rules for simplification of bracketing formulas as well as definitions for derived connectives as presented in the preliminaries apply here.

In practice, there are many variants for possible sets of justification terms, with some extensions and reductions of the set $Jt$ as defined above present. In general, a set of justification terms is expected to at least contain a countable set of variables and constants as well as to be closed under the $\cdot$  and $+$ operations. The $!$-operator, originating from the initial justification logic $\mathcal{LP}$, and the $?$-operator, relating to positive and negative introspection in explicit modal logics respectively, are of greater importance for extensions investigated later. There is, however, no disadvantage in allowing them right away.
\begin{definition}
A \emph{G\"odel justification (or G\"odel-Fitting) model} (over the language $\mathcal{L}_J$) is a structure $\mathfrak{M}=\langle W,R,\mathcal{E},e\rangle$ with
\begin{enumerate}
\item $W$ being a non-empty set, the domain $\mathcal{D}(\mathfrak{M})$ (of $\mathfrak{M}$), 
\item $R:W\times W\to [0,1]$,
\item $\mathcal{E}:W\times Jt\times\mathcal{L}_J\to [0,1]$,
\item $e:W\times Var\to [0,1]$,
\end{enumerate}
where $\mathcal{E}$ satisfies the closure conditions\footnote{These conditions represent natural generalizations of the classical conditions on boolean Fitting models, i.e. restricting $\mathcal{E}$ to $\{0,1\}$ returns them in a translated form.}
\begin{enumerate}[(i)]
\item $\mathcal{E}(w,t,\phi\rightarrow\psi)\odot\mathcal{E}(w,s,\phi)\leq\mathcal{E}(w,t\cdot s,\psi)$,
\item $\mathcal{E}(w,t,\phi)\oplus\mathcal{E}(w,s,\phi)\leq\mathcal{E}(w,t+s,\phi)$,
\end{enumerate}
for all $t,s\in Jt,\phi,\psi\in\mathcal{L}_J$ and $w\in W$.
\end{definition}
Note that, to simplify notation, we omitted the outer square brackets of justification terms inside of the evidence function $\mathcal{E}$ in the previous definition of a G\"odel justification model and continue to do so if the context is clear. The class of all G\"odel justification models is denoted by $\mathsf{GJ}$. We say that a G\"odel justification model is (simply) finite if its domain is finite.

These models are inspired by G\"odel-Kripke models, originally introduced in \cite{CR2009}, \cite{CR2010}, which form a similar fuzzy possible-world semantics for standard G\"odel modal logics.\footnote{The concept of many-valued Kripke models in the context of modal logics, especially with many-valued accessibility functions, was initiated by the work of Fitting in \cite{Fit1991}, \cite{Fit1992} where he studied a variant taking values in a finite lattice.}

We extend the evaluation function $e$ of a $\mathsf{GJ}$-model from $Var$ to the whole language $\mathcal{L}_J$ via the following inductive rules, for each world $w\in W$:
\begin{itemize}
\item $e(w,\bot)=0$,
\item $e(w,\phi\land\psi)=e(w,\phi)\odot e(w,\psi)$,
\item $e(w,\phi\rightarrow\psi)=e(w,\phi)\Rightarrow e(w,\psi)$,
\item $e(w,t:\phi)=\mathcal{E}(w,t,\phi)\odot\inf_{v\in W}\{R(w,v)\Rightarrow e(v,\phi)\}$.
\end{itemize}
\begin{remark}\label{rem:boxabuse}
As an abuse of notation, we write
\[
e(w,\Box\phi):=\inf_{v\in W}\{R(w,v)\Rightarrow e(v,\phi)\}
\]
in connection to standard G\"odel modal logic \cite{CR2010} although of course $\Box\phi$, that is the classical necessity-style operator $\Box$, in general, is not part of the underlying language. Following to this, we may rephrase the definition of the semantic evaluation of $t:\phi$ with $e(w,t:\phi)=\mathcal{E}(w,t,\phi)\odot e(w,\Box\phi)$.
\end{remark}
At a world $w$ in a $\mathsf{GJ}$-model $\mathfrak{M}=\langle W,R,\mathcal{E},e\rangle$, we may also extend $e(w,\cdot)$ to sets of formulas $\Gamma\subseteq\mathcal{L}_J$ with setting
\[
e(w,\Gamma):=\inf_{\psi\in\Gamma}\{e(w,\psi)\}.
\]
A $\mathsf{GJ}$-model $\mathfrak{M}=\langle W,R,\mathcal{E},e\rangle$ is called \emph{accessibility crisp} if $R$ is crisp, i.e. if $R(w,v)\in\{0,1\}$ for all $w,v\in W$. For a class of $\mathsf{GJ}$-models $\mathsf{C}$, we denote the subclass of all accessibility crisp models in $\mathsf{C}$ by $\mathsf{Cc}$. Similarly, as in standard G\"odel modal logics, we may now define the usual semantical notion of (\emph{local}) satisfiability in a model.
\begin{definition}
Let $\mathfrak{M}=\langle W,R,\mathcal{E},e\rangle$ be a $\mathsf{GJ}$-model, $\Gamma\cup\{\phi\}\subseteq\mathcal{L}_J$ and $w\in W$. We say
\begin{enumerate}[(i)]
\item $\mathfrak{M}$ \emph{satisfies} $\phi$ \emph{in} $w$, written $(\mathfrak{M},w)\models\phi$, iff $e(w,\phi)=1$,
\item $\phi$ \emph{is valid in} $\mathfrak{M}$, written $\mathfrak{M}\models\phi$, iff $\forall v\in W: e(v,\phi)=1$,
\end{enumerate}
and similarly for sets $\Gamma$
\begin{enumerate}[(i)]
\setcounter{enumi}{2}
\item $\mathfrak{M}$ \emph{satisfies} $\Gamma$ \emph{in} $w$, written $(\mathfrak{M},w)\models\Gamma$, iff $\forall\psi\in\Gamma:(\mathfrak{M},w)\models\psi$,
\item $\Gamma$ \emph{is valid in} $\mathfrak{M}$, written $\mathfrak{M}\models\Gamma$, iff $\forall\psi\in\Gamma:\mathfrak{M}\models\psi$.
\end{enumerate}
\end{definition}
This yields, similarly to the non-modal propositional case, two analogues for \emph{local} semantic inference in fuzzy Fitting models.
\begin{definition}
Let $\Gamma\cup\{\phi\}\subseteq\mathcal{L}_J$ and $\mathsf{C}$ a class of $\mathsf{GJ}$-models. We say that
\begin{enumerate}
\item $\Gamma$ \emph{entails} $\phi$ \emph{in} $\mathsf{C}$, written $\Gamma\models_{\mathsf{C}\leq}\phi$, if $\forall\mathfrak{M}=\langle W,R,\mathcal{E},e\rangle\in\mathsf{C}:\forall w\in W:e(w,\Gamma)\leq e(w,\phi)$,
\item $\Gamma$ \emph{$1$-entails} $\phi$ \emph{in} $\mathsf{C}$, written $\Gamma\models_\mathsf{C}\phi$, if $\forall\mathfrak{M}=\langle W,R,\mathcal{E},e\rangle\in\mathsf{C}:\forall w\in W:(\mathfrak{M},w)\models\Gamma$ implies $(\mathfrak{M},w)\models\phi$.
\end{enumerate}
\end{definition}
A formula $\phi$ is called \emph{$\mathsf{C}$-valid}, for a class of $\mathsf{GJ}$-models $\mathsf{C}$, if $\varnothing\models_\mathsf{C}\phi$. In this case, we also just write $\models_\mathsf{C}\phi$.
\begin{lemma}\label{lem:consequenceclassimpl}
For any class of $\mathsf{GJ}$-models $\mathsf{C}$ and any $\Gamma\cup\{\phi\}\subseteq\mathcal{L}_J$: $\Gamma\models_{\mathsf{C}\leq}\phi$ implies $\Gamma\models_\mathsf{C}\phi$.
\end{lemma}
\begin{proof}
Let $\mathsf{C}$ be a class of $\mathsf{GJ}$-models and assume $\Gamma\models_{\mathsf{C}\leq}\phi$. Thus,
\[
\forall\mathfrak{M}\in\mathsf{C}:\forall w\in\mathcal{D}(\mathfrak{M}):\inf_{\psi\in\Gamma}\{e(w,\psi)\}\leq e(w,\phi).
\]
Now, let $w\in\mathcal{D}(\mathfrak{M})$ for some $\mathfrak{M}\in\mathsf{C}$ and suppose $(\mathfrak{M},w)\models\Gamma$, i.e. $\forall\psi\in\nobreak\Gamma:e(w,\psi)=1$. Thus $\inf_{\psi\in\Gamma}\{e(w,\psi)\}=1$. By the above, we have $1=\inf_{\psi\in\Gamma}\{e(w,\psi)\}\leq e(w,\phi)\leq 1$, i.e. $e(w,\phi)=1$ and thus $(\mathfrak{M},w)\models\phi$. Thus $\Gamma\models_{\mathsf{C}}\phi$.
\end{proof}
Similarly, as in standard G\"odel modal logics, the converse of this statement will later follow from a strong completeness theorem for various model classes $\mathsf{C}$.\\

The following lemma and its proof are analogies to a similar statement in standard G\"odel modal logic \cite{CR2010}, where the authors proved it for the before mentioned G\"odel-Kripke models over a different language. For the notation here, however, we remind on Rem. \ref{rem:boxabuse}.
\begin{lemma}\label{lem:kaxiomvalid}
For any $\mathsf{GJ}$-model $\mathfrak{M}$, any $w\in\mathcal{D}(\mathfrak{M})$ and any $\phi,\psi\in\mathcal{L}_J$: $e(w,\Box(\phi\rightarrow\psi))\odot e(w,\Box\phi)\leq e(w,\Box\psi)$.
\end{lemma}
\begin{proof}
Let $\mathfrak{M}=\langle W,R,\mathcal{E},e\rangle$. We have for any $u\in W$ that
\begin{align*}
&\inf_{v\in W}\{R(w,v)\Rightarrow e(v,\phi\rightarrow\psi)\}\odot\inf_{v\in W}\{R(w,v)\Rightarrow e(v,\phi)\}\\ &\qquad\leq(R(w,u)\Rightarrow e(u,\phi\rightarrow\psi))\odot (R(w,u)\Rightarrow e(u,\phi))\\
&\qquad\leq R(w,u)\Rightarrow (e(u,\phi\rightarrow\psi)\odot e(u,\phi))\\
&\qquad\leq R(w,u)\Rightarrow e(u,\psi).
\end{align*}
Thus, by taking the infimum over $u$, we obtain
\[
e(w,\Box(\phi\rightarrow\psi))\odot e(w,\Box\phi)\leq\inf_{u\in W}\{R(w,u)\Rightarrow e(u,\psi)\}=e(w,\Box\psi).
\]
\end{proof}
By properties of $\odot$ and the residuum $\Rightarrow$, the result may also be rephrased as $e(w,\Box(\phi\rightarrow\psi))\leq e(w,\Box\phi)\Rightarrow e(w,\Box\psi)$ for any $\mathsf{GJ}$-model $\mathfrak{M}$ and $w\in\mathcal{D}(\mathfrak{M})$.
\begin{definition}
Let $\mathcal{GJ}_0$ be the following axiomatic extension, in the language $\mathcal{L}_J$, of the proof calculus for standard propositional G\"odel logic $\mathcal{G}$:
\begin{description}
\item [($P$)] The axiom schemes of the calculus $\mathcal{G}$,
\item [($J$)] $t:(\phi\rightarrow\psi)\rightarrow (s:\phi\rightarrow [t\cdot s]:\psi)$,
\item [($+$)] $t:\phi\rightarrow [t+s]:\phi$, $s:\phi\rightarrow [t+s]:\phi$,
\item [($MP$)] From $\phi$ and $\phi\rightarrow\psi$, infer $\psi$.
\end{description}
We denote inference of a formula $\phi\in\mathcal{L}_J$ from a set of formulas $\Gamma\subseteq\mathcal{L}_J$ in this calculus by $\Gamma\vdash_{\mathcal{GJ}_0}\phi$ (or $\Gamma\vdash\phi$ if the context is clear).
\end{definition}
\begin{proposition}\label{prop:j+valid}
The schemes ($J$) and ($+$) are $\mathsf{GJ}$-valid. 
\end{proposition}
\begin{proof}
Let $\mathfrak{M}=\langle W,R,\mathcal{E},e\rangle$ be a $\mathsf{GJ}$-model and $w\in W$.
\begin{description}
\item [($J$)] We have
\begin{align*}
&e(w,t:(\phi\rightarrow\psi))\odot e(w,s:\phi)\\
&\qquad =(\mathcal{E}(w,t,\phi\rightarrow\nobreak\psi)\odot e(w,\Box(\phi\rightarrow\psi)))\odot(\mathcal{E}(w,s,\phi)\odot e(w,\Box\phi))\\
&\qquad=(\mathcal{E}(w,t,\phi\rightarrow\psi)\odot\mathcal{E}(w,s,\phi))\odot (e(w,\Box(\phi\rightarrow\psi))\odot e(w,\Box\phi))
\end{align*}
by commutativity and associativity of $\odot$. As
\[
\mathcal{E}(w,t,\phi\rightarrow\psi)\odot\mathcal{E}(w,s,\phi)\leq\mathcal{E}(w,t\cdot s,\psi)
\]
by property (i) on $\mathcal{E}$ of a $\mathsf{GJ}$-model and
\[
e(w,\Box(\phi\rightarrow\psi))\odot e(w,\Box\phi)\leq e(w,\Box\psi)
\]
by Lem. \ref{lem:kaxiomvalid}, we have through monotonicity of $\odot$: 
\begin{align*}
&(\mathcal{E}(w,t,\phi\rightarrow\psi)\odot\mathcal{E}(w,s,\phi))\odot (e(w,\Box(\phi\rightarrow\psi))\odot e(w,\Box\phi))\\
&\qquad\leq\mathcal{E}(w,t\cdot s,\psi)\odot e(w,\Box\psi)=e(w,[t\cdot s]:\psi).
\end{align*}
Thus, $e(w,t:(\phi\rightarrow\psi))\odot e(w,s:\phi)\leq e(w,[t\cdot s]:\psi)$, i.e. by properties of the residuum \[
e(w,t:(\phi\rightarrow\psi))\leq e(w,s:\phi)\Rightarrow e(w,[t\cdot s]:\psi)=e(w,s:\phi\rightarrow [t\cdot s]:\psi).\]
\item [($+$)] We just show the first case, as the second case follows similarly. We have $e(w,t:\phi)=\mathcal{E}(w,t,\phi)\odot e(w,\Box\phi)$. By
\[
\mathcal{E}(w,t,\phi)\oplus\mathcal{E}(w,s,\phi)\leq\mathcal{E}(w,t+s,\phi),
\]
as of property (ii) on $\mathcal{E}$ of a $\mathsf{GJ}$-model, we have $\mathcal{E}(w,t,\phi)\leq\mathcal{E}(w,t+s,\phi)$. Thus again by monotonicity of $\odot$, we have
\[
\mathcal{E}(w,t,\phi)\odot e(w,\Box\phi)\leq\mathcal{E}(w,t+s,\phi)\odot e(w,\Box\phi)=e(w,[t+s]:\phi).
\]
\end{description}
\end{proof}
\subsection{Constant specifications and internalization}
Constant specifications are a weakened implementation of the principle of logical awareness, i.e. regarding axioms to be self-evidently justified, with weakened in the sense that we may restrict this view to a corresponding subset of the axioms in question. From a basic practical point, a constant specification helps an agent to make more justified inference.
\begin{definition}
For a given proof calculus $\mathcal{S}$, defined over the corresponding language $\mathcal{L}_J$, a \emph{constant specification for $\mathcal{S}$} is a set $CS$ of formulas of the form
\[
c_{i_n}:c_{i_{n-1}}:\dots:c_{i_1}:\phi
\]
where $n\geq 1$, $\phi$ is an axiom instance of $\mathcal{S}$ and the $c_{i_k}$'s are constants. Additionally, a constant specification is expected to be downward closed, i.e.
\[
\text{if }c_{i_n}:c_{i_{n-1}}:\dots:c_{i_1}:\phi\in CS\text{, then }c_{i_k}:\dots:c_{i_1}:\phi\in CS
\] 
for all $k=1,\dots,n$.

We call $CS$ \emph{axiomatically appropriate} for $\mathcal{S}$, if for each axiom instance $\phi$ of $\mathcal{S}$, there is a constant $c\in C$ such that $c:\phi\in CS$, and if $c_{i_n}:c_{i_{n-1}}:\dots:c_{i_1}:\phi\in CS$, then $c_{i_{n+1}}:c_{i_n}:c_{i_{n-1}}:\dots:c_{i_1}:\phi\in CS$ for some constant $c_{i_{n+1}}$.
\end{definition}
\begin{definition}
We say that a G\"odel justification model $\mathfrak{M}=\langle W,R,\mathcal{E},e\rangle$ \emph{respects a constant specification} $CS$, if
\[
\forall c:\phi\in CS:\forall w\in W:\mathcal{E}(w,c,\phi)=1.
\]
For a class $\mathsf{C}$ of $\mathsf{GJ}$-models, we denote the subclass of all $\mathsf{GJ}$-models in $\mathsf{C}$ respecting a constant specification $CS$ by $\mathsf{C_{CS}}$.
\end{definition}
\begin{definition}
Let $CS$ be a constant specification (for $\mathcal{GJ}_0$). We define $\mathcal{GJ}_{CS}$ as $\mathcal{GJ}_0$ extended by the rule
\[
(CS):\text{ From }c:\phi\in CS\text{, infer }c:\phi.
\]
\end{definition}
Clearly, $\mathcal{GJ}_0$ relates to $\mathcal{GJ}_\varnothing$. Similarly, as propositional G\"odel logic, G\"odel justification logic enjoys the classical deduction theorem, which is a notable exception in comparison to other representatives in the framework of fuzzy (justification) logics.
\begin{lemma}[Deduction theorem]
Let $\Gamma\cup\{\alpha,\phi\}\subseteq\mathcal{L}_J$: $\Gamma\cup\{\alpha\}\vdash_{\mathcal{GJ}_{CS}}\phi$ iff $\Gamma\vdash_{\mathcal{GJ}_{CS}}\alpha\rightarrow\phi$.
\end{lemma}
The proof is essentially the same as in the case of classical (justification) logic, proceeding via the usual induction which employs Lem. \ref{lem:goedellogicthms}, and is thus omitted.
\begin{lemma}\label{lem:csrulevalid}
Every formula that is deduced by the rule ($CS$) is valid in the model class $\mathsf{GJ_{CS}}$.
\end{lemma}
\begin{proof}
Let $c_{i_n}:\dots:c_{i_1}:\phi\in CS$ and let $\mathfrak{M}=\langle W,R,\mathcal{E},e\rangle$ be a $\mathsf{GJ}$-model respecting $CS$. Now, as $CS$ is downward closed, we have $c_{i_k}:\dots:c_{i_1}:\phi\in CS$ for every $k\in\{1,\dots,n\}$. Thus, for every $k\in\{1,\dots,n\}$, we have 
\[
\mathcal{E}(w,c_{i_k},c_{i_{k-1}}:\dots:c_{i_1}:\phi)=1
\] 
for all $w\in W$. As $\phi$ is an axiom of $\mathcal{GJ}_{CS}$, we have, as all axioms are $\mathsf{GJ}$-valid, that $e(w,\phi)=1$ for all $w\in W$. Thus, $e(w,\Box\phi)=1$ for all $w\in W$ and thus $e(w,c_{i_1}:\phi)=\mathcal{E}(w,c_{i_1},\phi)\odot e(w,\Box\phi)=1$ for all $w\in W$. From this, we have that $e(w,c_{i_2}:c_{i_1}:\phi)=\mathcal{E}(w,c_{i_2},c_{i_1}:\phi)\odot e(w,\Box c_{i_1}:\phi)=1$. Continuing this up to $n$ gives
\[
\mathcal{E}(w,c_{i_n},c_{i_{n-1}}:\dots:c_{i_1}:\phi)\odot e(w,\Box c_{i_{n-1}}:\dots:c_{i_1}:\phi)=1
\]
for all $w\in W$.
\end{proof}
\begin{definition}
We say that $\mathcal{GJ}_{CS}$ \emph{enjoys internalization}, if $\vdash_{\mathcal{GJ}_{CS}}\phi$ implies that there exists a justification term $t\in Jt$ such that $\vdash_{\mathcal{GJ}_{CS}}t:\phi$.
\end{definition}
\begin{lemma}[Lifting lemma]\label{lem:lifting}
Let $CS$ be an axiomatically appropriate constant specification for $\mathcal{GJ}_0$. If $\{\psi_1,\dots,\psi_n\}\vdash_{\mathcal{GJ}_{CS}}\phi$, then for any justification terms $t_1,\dots,t_n\in Jt$ there is a justification term $t\in Jt$ such that \[\{t_1:\psi_1,\dots,t_n:\psi_n\}\vdash_{\mathcal{GJ}_{CS}}t:\phi.\]
\end{lemma}
The proof of this lemma is strictly similar to the proof in the classical case (see e.g. \cite{Art2001}, \cite{Kuz2008}) and thus omitted here. The following is a direct consequence of the Lifting lemma.
\begin{lemma}
If $CS$ is an axiomatically appropriate constant specification for $\mathcal{GJ}_0$, then $\mathcal{GJ}_{CS}$ enjoys internalization.
\end{lemma}
Using the deduction theorem, we may now obtain the soundness of the system $\mathcal{GJ}_{CS}$ for any constant specification $CS$ for $\mathcal{GJ}_0$.
\begin{lemma}[Soundness of $\mathcal{GJ}_{CS}$]\label{lem:gjcssoundness}
For any $\Gamma\cup\{\phi\}\subseteq\mathcal{L}_J$: $\Gamma\vdash_{\mathcal{GJ}_{CS}}\phi$ implies $\Gamma\models_{\mathsf{GJ_{CS}}\leq}\phi$.
\end{lemma}
\begin{proof}
We have that $\Gamma\vdash_{\mathcal{GJ}_{CS}}\phi$ implies $\{\psi_1,\dots,\psi_n\}\vdash_{\mathcal{GJ}_{CS}}\phi$ for some $\{\psi_1,\dots,\psi_n\}\subseteq\Gamma$. By repeated application of the deduction theorem and using axiom ($A5a$), we have $\vdash_{\mathcal{GJ}_{CS}}\bigwedge_{i=1}^n\psi_i\rightarrow\phi$. As of Prop. \ref{prop:j+valid} and Thm. \ref{thm:gssc}, all axioms of $\mathcal{GJ}_{CS}$ are $\mathsf{GJ_{CS}}$-valid. Of course ($MP$), and as of Lem. \ref{lem:csrulevalid}, also ($CS$) preserve validity (in $\mathsf{GJ_{CS}}$). Thus $\models_\mathsf{GJ_{CS}}\bigwedge_{i=1}^n\psi_i\rightarrow\phi$ and therefore, for any $\mathsf{GJ_{CS}}$-model $\mathfrak{M}$ and any $w\in\mathcal{D}(\mathfrak{M})$, we have $(\mathfrak{M},w)\models\bigwedge_{i=1}^n\psi_i\rightarrow\phi$, i.e. $e(w,\Gamma)\leq e(w,\bigwedge_{i=1}^n\psi_i)\leq e(w,\phi)$, as $\{\psi_1,\dots,\psi_n\}\subseteq\Gamma$, and thus $\Gamma\models_{\mathsf{GJ_{CS}}\leq}\phi$.
\end{proof}
\section{Modal-type extensions}
Similar to the realm of classical modal logic, the framework of classical justification logic spreads out over numerous extensions of the basic axiomatic system for justifications (similar to $\mathcal{GJ}_0$ here). Of mainline importance are here explicit justification formulas standing in analogy to classical unexplicit epistemic (modal) principles like truth and positive introspection, etc. In this section, we present analogue extensions in the context of \emph{fuzzy} justification logic, both model-theoretically, by characterizing the fuzzy versions of the associated Fitting models, and axiomatically. We do not go into surrounding (philosophical) detail about the here studied principles, however, for an exposition in the classical case, refer to \cite{Art2008}.
\subsection{Factivity}
\begin{definition}
We define $\mathcal{GJT}_0$ as the axiomatic extension of $\mathcal{GJ}_0$ by the axiom scheme $(F): t:\phi\rightarrow\phi$.
\end{definition}
For a constant specification $CS$ for $\mathcal{GJT}_0$, we write $\mathcal{GJT}_{CS}$ for the calculus $\mathcal{GJT}_0$ together with the constant specification rule $(CS)$. 
\begin{definition}
A G\"odel justification model $\mathfrak{M}=\langle W,R,\mathcal{E},e\rangle$ is called \emph{reflexive}, if $\forall w\in W: R(w,w)=1$. The class of all reflexive $\mathsf{GJ}$-models is denoted by $\mathsf{GJT}$.
\end{definition}
\begin{proposition}\label{prop:factivityvalid}
The scheme $t:\phi\rightarrow\phi$ is valid in the class $\mathsf{GJT}$.
\end{proposition}
\begin{proof}
Let $\mathfrak{M}=\langle W,R,\mathcal{E},e\rangle\in\mathsf{GJT}$ and let $w\in W$. Then
\begin{align*}
e(w,t:\phi)&=\mathcal{E}(w,t,\phi)\odot\inf_{v\in W}\{R(w,v)\Rightarrow e(v,\phi)\}\\
           &\leq R(w,w)\Rightarrow e(w,\phi)=e(w,\phi)
\end{align*}
where the last equality follows from $R(w,w)=1$ for all $w\in W$, as $\mathfrak{M}$ is reflexive.
\end{proof}
Similarly, as before, we obtain the soundness of $\mathcal{GJT}_{CS}$ (for some $CS$ for $\mathcal{GJT}_0$) w.r.t. its intended model class and the proof is thus omitted here.
\begin{lemma}[Soundness of $\mathcal{GJT}_{CS}$]\label{lem:gjtcssoundness}
For any $\Gamma\cup\{\phi\}\subseteq\mathcal{L}_J$: $\Gamma\vdash_{\mathcal{GJT}_{CS}}\phi$ implies $\Gamma\models_{\mathsf{GJT_{CS}}\leq}\phi$.
\end{lemma}
\subsection{Positive introspection}
\begin{definition}
We define the following extensions of $\mathcal{GJ}_0$:
\begin{enumerate}
\item $\mathcal{GJ}4_0:=\mathcal{GJ}_0+(PI):t:\phi\rightarrow !t:t:\phi$,
\item $\mathcal{GLP}_0:=\mathcal{GJT}_0+(PI):t:\phi\rightarrow !t:t:\phi$.
\end{enumerate}
\end{definition}
For a constant specification $CS$ for $\mathcal{GJ}4_0$ or $\mathcal{GLP}_0$, we again write $\mathcal{GJ}4_{CS}$ or $\mathcal{GLP}_{CS}$ for the respective extensions by the rule $(CS)$. We can now find similar fuzzy analogues to the classical additional properties of Fitting models regarding positive introspection.
\begin{definition}\label{def:j4lpmodels}
With $\mathsf{GJ4}$, we denote that class of G\"odel justification models $\mathfrak{M}=\langle W,R,\mathcal{E},e\rangle$ satisfying
\begin{enumerate}[(i)]
\item $\mathcal{E}(w,t,\phi)\odot R(w,v)\leq\mathcal{E}(v,t,\phi)$ for all $t\in Jt,\phi\in\mathcal{L}_J,w,v\in W$ \\(monotonicity of $\mathcal{E}$ w.r.t. $R$),
\item $R(w,v)\odot R(v,u)\leq R(w,u)$ for all $w,v,u\in W$ \\((min-)transitivity of $R$),
\item $\mathcal{E}(w,t,\phi)\leq\mathcal{E}(w,!t,t:\phi)$ for all $t\in Jt,\phi\in\mathcal{L}_J,w\in W$ \\(positive introspectivity of $\mathcal{E}$).
\end{enumerate}
The subclass of all \emph{reflexive} $\mathsf{GJ4}$-models is denoted respectively with $\mathsf{GLP}$.
\end{definition}
\begin{lemma}\label{lem:modaltransitivity}
In a (min-)transitive $\mathsf{GJ}$-model $\mathfrak{M}=\langle W,R,\mathcal{E},e\rangle$, it holds for any $w,v\in W$ and any $\phi\in\mathcal{L}_J$ that $e(w,\Box\phi)\leq R(w,v)\Rightarrow e(v,\Box\phi)$.
\end{lemma}
\begin{proof}
We have that for any model $\mathfrak{M}=\langle W,R,\mathcal{E},e\rangle$ and any $w,v,u\in W$ that $R(w,v)\odot R(w,u)\leq R(w,u)$ and thus
\begin{align*}
(e(w,\Box\phi)\odot R(w,v))\odot R(v,u) &=e(w,\Box\phi)\odot (R(w,v)\odot R(v,u))\\
                                        &\leq (R(w,u)\Rightarrow e(u,\phi))\odot R(w,u)\\
                                        &\leq e(u,\phi)
\end{align*}
and thus $e(w,\Box\phi)\odot R(w,v)\leq R(v,u)\Rightarrow e(u,\phi)$ by properties of the residuum. As $u$ was arbitrary, we may take the infimum over $u$, obtaining $e(w,\Box\phi)\odot R(w,v)\leq e(v,\Box\phi)$.
\end{proof}
\begin{proposition}
The scheme $t:\phi\rightarrow !t:t:\phi$ is valid in the class $\mathsf{GJ4}$.
\end{proposition}
\begin{proof}
Let $\mathfrak{M}=\langle W,R,\mathcal{E},e\rangle$ be a $\mathsf{GJ4}$-model and $w\in W$. Now, we have
\[
\mathcal{E}(w,t,\phi)\leq R(w,v)\Rightarrow\mathcal{E}(v,t,\phi)
\]
by monotonicity of $\mathcal{E}$ over $R$ and properties of the residuum for every $v\in W$. By Lem. \ref{lem:modaltransitivity} and monotonicity of $\odot$, we have thus
\begin{align*}
e(w,t:\phi)&=\mathcal{E}(w,t,\phi)\odot e(w,\Box\phi)\\
           &\leq (R(w,v)\Rightarrow\mathcal{E}(v,t,\phi))\odot (R(w,v)\Rightarrow e(v,\Box\phi))
\end{align*}
for all $v\in W$, i.e. we have
\begin{align*}
e(w,t:\phi)&\leq\inf_{v\in W}\{(R(w,v)\Rightarrow\mathcal{E}(v,t,\phi))\odot (R(w,v)\Rightarrow e(v,\Box\phi))\}\\
           &=\inf_{v\in W}\{R(w,v)\Rightarrow (\mathcal{E}(v,t,\phi)\odot e(v,\Box\phi))\}
\end{align*}
and thus, we have $e(w,t:\phi)\leq\inf_{v\in W}\{R(w,v)\Rightarrow e(v,t:\phi)\}$. Similarly, we have
\[
e(w,t:\phi)=\mathcal{E}(w,t,\phi)\odot e(w,\Box\phi)\leq\mathcal{E}(w,!t,t:\phi)
\]
by positive introspectivity and properties of $\odot$. Thus, finally
\[
e(w,t:\phi)\leq\mathcal{E}(w,!t,t:\phi)\odot\inf_{v\in W}\{R(w,v)\Rightarrow e(v,t:\phi)\}=e(w,!t:t:\phi)
\]
and therefore $e(w,t:\phi\rightarrow !t:t:\phi)=1$.
\end{proof}
We again obtain a soundness result for those proof systems in the same way as before, for any well-defined constant specification $CS$.
\begin{lemma}[Soundness of $\mathcal{GJ}4_{CS}, \mathcal{GLP}_{CS}$]\label{lem:j4lpcssoundness}
For any $\Gamma\cup\{\phi\}\subseteq\mathcal{L}_J$, we have
\begin{enumerate}
\item $\Gamma\vdash_{\mathcal{GJ}4_{CS}}\phi$ implies $\Gamma\models_{\mathsf{GJ4_{CS}}\leq}\phi$,
\item $\Gamma\vdash_{\mathcal{GLP}_{CS}}\phi$ implies $\Gamma\models_{\mathsf{GLP}_{CS}\leq}\phi$.
\end{enumerate}
\end{lemma}
\subsection{Negative introspection}
\begin{definition}
We define the following extensions of $\mathcal{GJ}4_0$:
\begin{enumerate}
\item $\mathcal{GJ}45_0:=\mathcal{GJ}4_0 +(NI): \neg t:\phi\rightarrow ?t:\neg t:\phi$,
\item $\mathcal{GJT}45_0:=\mathcal{GJ}45_0 +(F):t:\phi\rightarrow\phi$.
\end{enumerate}
\end{definition}
For $CS$ being a constant specification for $\mathcal{GJ}45_0$ or $\mathcal{GJT}45_0$, we again write $\mathcal{GJ}45_{CS}$ or $\mathcal{GJT}45_{CS}$ for the respective extensions by the rule $(CS)$.
\begin{definition}\label{def:45models}
The class of $\mathsf{GJ4}$-models satisfying
\begin{enumerate}[(i)]
\item $\sim\mathcal{E}(w,t,\phi)\leq\mathcal{E}(w,?t,\neg t:\phi)$ for all $t\in Jt,\phi\in\mathcal{L}_J,w\in W$ \\(negative introspectivity of $\mathcal{E}$),
\item $\mathcal{E}(w,t,\phi)\leq e(w,t:\phi)$ for all $t\in Jt,\phi\in\mathcal{L}_J,w\in W$ \\(strong evidence),
\end{enumerate}
is denoted by $\mathsf{GJ45}$. The class of all $\mathsf{GJ45}$-model with \emph{reflexive} accessibility function is in the following denoted by $\mathsf{GJT45}$.
\end{definition}
\begin{proposition}
The scheme $\neg t:\phi\rightarrow ?t:\neg t:\phi$ is valid in the model class $\mathsf{GJ45}$.
\end{proposition}
\begin{proof}
Let $\mathfrak{M}=\langle W,R,\mathcal{E},e\rangle$ be a $\mathsf{GJ45}$-model. We have that
\begin{align*}
e(w,\neg t:\phi)=\sim e(w,t:\phi)&\leq\;\sim\mathcal{E}(w,t,\phi)     &&\text{(by strong evid., Lem. \ref{lem:mintnormprop})}\\
                                 &\leq\;\mathcal{E}(w,?t,\neg t:\phi) &&\text{(by neg. intro.)}\\
                                 &\leq\; e(w,?t:\neg t:\phi)          &&\text{(by strong evid.)}
\end{align*}
for any $t\in Jt,\phi\in\mathcal{L}_J,w\in W$.
\end{proof}
\begin{lemma}[Soundness of $\mathcal{GJ}45_{CS}, \mathcal{GJT}45_{CS}$]\label{lem:45soundness}
For any $\Gamma\cup\{\phi\}\subseteq\mathcal{L}_J$, we have
\begin{enumerate}
\item $\Gamma\vdash_{\mathcal{GJ}45_{CS}}\phi$ implies $\Gamma\models_{\mathsf{GJ45_{CS}}\leq}\phi$,
\item $\Gamma\vdash_{\mathcal{GJT}45_{CS}}\phi$ implies $\Gamma\models_{\mathsf{GJT45_{CS}}\leq}\phi$.
\end{enumerate}
\end{lemma}
\section{Completeness}
In the following, let $\mathcal{GJL}_0$ be one of the previously introduced systems of G\"odel justification logic, i.e.
\[
\mathcal{GJL}_0\in\{\mathcal{GJ}_0,\mathcal{GJT}_0,\mathcal{GJ}4_0,\mathcal{GLP}_0,\mathcal{GJ}45_0,\mathcal{GJT}45_0\}
\]
and let $CS$ be a constant specification for $\mathcal{GJL}_0$. With $\mathsf{GJL_{CS}}$, we denote the corresponding class of G\"odel justification models (respecting the given constant specification) for which we have established soundness.
\begin{definition}\label{def:startranslation}
We define the language $\mathcal{L}_0^\star:=\mathcal{L}_0(Var^\star)$ where
\[
Var^\star:=Var\cup\{\phi_t\mid\phi\in\mathcal{L}_J, t\in Jt\}.
\]
The translation function $\star:\mathcal{L}_J\to\mathcal{L}_0^\star$ is defined inductively as follows:
\begin{enumerate}[(i)]
\item $\bot\mapsto\bot$,
\item $p\mapsto p$, $p\in Var$,
\item $(\phi\land\psi)\mapsto\phi^\star\land\psi^\star$,
\item $(\phi\rightarrow\psi)\mapsto\phi^\star\rightarrow\psi^\star$,
\item $t:\phi\mapsto\phi_t$.
\end{enumerate}
\end{definition}
\begin{lemma}
$\star$ is a bijection.
\end{lemma}
The concrete proof is omitted here, however surjectivity follows almost directly by considering a formula in $\mathcal{L}_0^\star$ and replacing every $\phi_t$ by $t:\phi$. The resulting formula lies in $\mathcal{L}_J$ and has the expected translation as $\star$ distributes over all propositional connectives. Injectivity follows from a simple induction on the structure of the formulas.

$\star$ is naturally extended to sets of formulas via $\Gamma^\star:=\{\phi^\star\mid\phi\in\Gamma\}$.\footnote{Note, that $\mathcal{L}_0^\star$ and similarly $Var^\star$ from Def. \ref{def:startranslation} are an abuse of notation and do not denote the $\star$-translation of $\mathcal{L}_0$ and $Var$ respectively. For this, one may write $(\mathcal{L}_0)^\star$ or $(Var)^\star$.} We denote by $Th_{\mathcal{GJL}_{CS}}$ the set of theorems of the calculus $\mathcal{GJL}_{CS}$ for some constant specification $CS$, i.e. $Th_{\mathcal{GJL}_{CS}}:=\{\phi\in\mathcal{L}_J\mid\;\vdash_{\mathcal{GJL}_{CS}}\phi\}$.
\begin{lemma}\label{lem:modalproptransinv}
Let $\Gamma\cup\{\phi\}\subseteq\mathcal{L}_J$. Then $\Gamma\vdash_{\mathcal{GJL}_{CS}}\phi\text{ iff }\Gamma^\star\cup (Th_{\mathcal{GJL}_{CS}})^\star\vdash_\mathcal{G}\phi^\star$.
\end{lemma}
\begin{proof}
Let $\Gamma\subseteq\mathcal{L}_J$ be arbitrary.
\begin{description}
\item [$\Rightarrow$] We prove by strong induction on $k$ that, for any $\phi\in\mathcal{L}_J$, if $\Gamma\vdash_{\mathcal{GJL}_{CS}}\phi$ with a proof of length $k$, then $\Gamma^\star\cup(Th_{\mathcal{GJL}_{CS}})^\star\vdash_\mathcal{G}\phi^\star$.
\begin{description}
\item [(IB)] Let $\phi\in\mathcal{L}_J$ be arbitrary. Suppose $\Gamma\vdash_{\mathcal{GJL}_{CS}}\phi$ with a proof of length $1$, then either
\begin{enumerate}[(i)]
\item $\phi\in\Gamma$,
\item $\phi$ is an axiom instance of $\mathcal{GJL}_{CS}$, or
\item $\phi$ was obtained by ($CS$).
\end{enumerate}
For (i), we have $\phi^\star\in\Gamma^\star$, from which the claim follows. For (ii) and (iii), we have that $\phi\in Th_{\mathcal{GJL}_{CS}}$, i.e. $\phi^\star\in (Th_{\mathcal{GJL}_{CS}})^\star$, which gives the claim.
\item [(IS)] Let $k\geq 1$. Suppose that for all $\chi\in\mathcal{L}_J$, if $\Gamma\vdash_{\mathcal{GJL}_{CS}}\chi$ with a proof of length $l\leq k$, then $\Gamma^\star\cup (Th_{\mathcal{GJL}_{CS}})^\star\vdash_\mathcal{G}\chi^\star$. Let $\phi\in\mathcal{L}_J$ be arbitrary and suppose that $\Gamma\vdash_{\mathcal{GJL}_{CS}}\phi$ with a proof of length $k+1$. Then either $\phi$ was obtained as in \textbf{(IB)}, in which case we can proceed similarly, or $\phi$ was obtained by ($MP$), i.e. we have $\Gamma\vdash_{\mathcal{GJL}_{CS}}\psi$ and $\Gamma\vdash_{\mathcal{GJL}_{CS}}\psi\rightarrow\phi$ for some $\psi\in\mathcal{L}_J$. By the induction hypothesis, as they have shorter proofs, we have $\Gamma^\star\cup (Th_{\mathcal{GJL}_{CS}})^\star\vdash_\mathcal{G}\psi^\star$ and $\Gamma^\star\cup (Th_{\mathcal{GJL}_{CS}})^\star\vdash_\mathcal{G}\psi^\star\rightarrow\phi^\star$. Thus, by ($MP$) in $\mathcal{G}$, we deduce $\Gamma^\star\cup (Th_{\mathcal{GJL}_{CS}})^\star\vdash_\mathcal{G}\phi^\star$.
\end{description}
\item [$\Leftarrow$] We again show by a strong induction on $k$ that, for any $k\geq 1$ and for any $\phi\in\mathcal{L}_J$, if $\Gamma^\star\cup (Th_{\mathcal{GJL}_{CS}})^\star\vdash_\mathcal{G}\phi^\star$ with a proof of length $k$, then $\Gamma\vdash_{\mathcal{GJL}_{CS}}\phi$.
\begin{description}
\item [(IB)] Let $\phi\in\mathcal{L}_J$ be arbitrary such that $\Gamma^\star\cup (Th_{\mathcal{GJL}_{CS}})^\star\vdash_\mathcal{G}\phi^\star$ has a proof of length $1$. Then either
\begin{enumerate}[(i)]
\item $\phi^\star\in\Gamma^\star$,
\item $\phi^\star\in(Th_{\mathcal{GJL}_{CS}})^\star$, or
\item $\phi^\star$ is an axiom instance in $\mathcal{G}$.
\end{enumerate}
For (i), we have that $\phi\in\Gamma$, i.e. $\Gamma\vdash_{\mathcal{GJL}_{CS}}\phi$, while for (ii), we have that $\phi\in Th_{\mathcal{GJL}_{CS}}$ and thus $\vdash_{\mathcal{GJL}_{CS}}\phi$. Finally, if $\phi^\star$ is an axiom instance in $\mathcal{G}$, then $\phi$, resulting by replacing every occurrence of some $\psi_t$ by $t:\psi$, is an instance of the same axiom in $\mathcal{GJL}_{CS}$.
\item [(IS)] For $k\geq 1$, suppose that for all $\chi\in\mathcal{L}_J$, if $\Gamma^\star\cup (Th_{\mathcal{GJL}_{CS}})^\star\vdash_\mathcal{G}\chi^\star$ with a proof of length $l\leq k$, then $\Gamma\vdash_{\mathcal{GJL}_{CS}}\chi$. Suppose that $\Gamma^\star\cup (Th_{\mathcal{GJL}_{CS}})^\star\vdash_\mathcal{G}\phi^\star$ has a proof of length $k+1$ for an arbitrary $\phi\in\mathcal{L}_J$. Again $\phi^\star$ may have been obtained as in \textbf{(IB)}, where we proceed as shown. Otherwise, $\phi^\star$ was again obtained by ($MP$), i.e. $\Gamma^\star\cup (Th_{\mathcal{GJL}_{CS}})^\star\vdash_\mathcal{G}\psi^\star\rightarrow\phi^\star$ and $\Gamma^\star\cup (Th_{\mathcal{GJL}_{CS}})^\star\vdash_\mathcal{G}\psi^\star$ for some $\psi\in\mathcal{L}_J$ as $\star$ is bijective between $\mathcal{L}_J$ and $\mathcal{L}_0^\star$. By the definition of $\star$, we have $\Gamma^\star\cup (Th_{\mathcal{GJL}_{CS}})^\star\vdash_\mathcal{G}(\psi\rightarrow\phi)^\star$. By the induction hypothesis, as the corresponding proofs are shorter, we have $\Gamma\vdash_{\mathcal{GJL}_{CS}}\psi$ and $\Gamma\vdash_{\mathcal{GJL}_{CS}}\psi\rightarrow\phi$, i.e. by ($MP$) in $\mathcal{GJL}_{CS}$ we deduce $\Gamma\vdash_{\mathcal{GJL}_{CS}}\phi$.
\end{description}
\end{description}
\end{proof}
\begin{definition}[Canonical model for $\mathcal{GJL}_{CS}$]
The \emph{canonical model for} $\mathcal{GJL}_{CS}$, $\mathfrak{M}^c(\mathcal{GJL}_{CS})=\langle W^c,R^c,\mathcal{E}^c,e^c\rangle$, is defined as follows:
\begin{enumerate}
\item $W^c:=\{v\in\mathsf{Ev}(\mathcal{L}_0^\star)\mid v((Th_{\mathcal{GJL}_{CS}})^\star)=1\}$,
\item $R^c(v,w):=\begin{cases}1, &\text{if }\forall\phi\in\mathcal{L}_J:\forall t\in Jt:v(\phi_t)\leq w(\phi^\star)\\0, &\text{otherwise}\end{cases}$ for all $v,w\in W^c$,
\item $\mathcal{E}^c(v,t,\phi):=v(\phi_t)$ for all $v\in W^c$, $t\in Jt$ and $\phi\in\mathcal{L}_J$,
\item $e^c(v,p):=v(p)$ for all $v\in W^c, p\in Var$.
\end{enumerate}
\end{definition}
$e^c$ is extended from $Var$ to $\mathcal{L}_J$ as before.
\begin{lemma}[Truth lemma]\label{lem:gjlcscanonmodtruth}
Let $\mathfrak{M}^c(\mathcal{GJL}_{CS})=\langle W^c,R^c,\mathcal{E}^c,e^c\rangle$. For all $\phi\in\mathcal{L}_J$ and any $v\in W^c$: $e^c(v,\phi)=v(\phi^\star)$.
\end{lemma}
\begin{proof}
Induction on the structure of the formula $\phi$. 
\begin{description}
\item [(IB)] For the base case of $\phi=p\in Var$, we have $e^c(v,p)=v(p)=v(p^\star)$ for all $v\in W^c$. Similarly, for $\phi=\bot$, we have $e^c(v,\bot)=0=v(\bot)=v(\bot^\star)$.
\item [(IS)] For the induction step, we divide between the different connectives. 

We have, for $\phi=\psi\rightarrow\chi$, that $e^c(v,\psi\rightarrow\chi) =e^c(v,\psi)\Rightarrow e^c(v,\chi)=v(\psi^\star)\Rightarrow v(\chi^\star)=v(\psi^\star\rightarrow\chi^\star)=v((\psi\rightarrow\chi)^\star)$ straightforwardly by \textbf{(IH)}. Similarly, for $\phi=\psi\land\chi$, we obtain $e^c(v,\psi\land\chi)=v((\psi\land\chi)^\star)$ by \textbf{(IH)} as well.\\

Hence, we are left with showing that $e^c(v,t:\psi)=v(\psi_t)$ for an arbitrary $v\in W^c$.
As $\mathcal{E}^c(v,t,\psi)=v(\psi_t)$ per definition, it holds that
\[
e^c(v,t:\psi)=\mathcal{E}^c(v,t,\psi)\odot e^c(v,\Box\psi)=v(\psi_t)\odot e^c(v,\Box\psi).
\]
Thus, it suffices to show that $e^c(v,\Box\psi)\geq v(\psi_t)$ by the laws of $\odot=\min$. Now, by \textbf{(IH)}, we have $e^c(w,\psi)=w(\psi^\star)$ for any $w\in W^c$ and thus
\[
e^c(v,\Box\psi)=\inf_{w\in W^c}\{R^c(v,w)\Rightarrow e^c(w,\psi)\}=\inf_{w\in W^c}\{R^c(v,w)\Rightarrow w(\psi^\star)\}.
\]
As $R^c$ is crisp, we now have
\[
\inf_{w\in W^c}\{R^c(v,w)\Rightarrow w(\psi^\star)\}=\inf\{w(\psi^\star)\mid w\in W^c, R^c(v,w)=1\}.
\]
Now, for all $w\in W^c$ such that $R^c(v,w)=1$, we have $v(\psi_t)\leq w(\psi^\star)$, i.e.
\[
v(\psi_t)\leq\inf\{w(\psi^\star)\mid R^c(v,w)=1\}=e^c(v,\Box\psi).
\]
Thus $e^c(v,t:\psi)=v(\psi_t)\odot e^c(v,\Box\psi)=\min\{v(\psi_t),e^c(v,\Box\psi)\}=v(\psi_t)$.
\end{description}
\end{proof}
\begin{lemma}\label{lem:canonmodstrongevid}
$\mathfrak{M}^c(\mathcal{GJL}_{CS})=\langle W^c,R^c,\mathcal{E}^c,e^c\rangle$ has the strong evidence property, i.e.
\[
\mathcal{E}^c(v,t,\phi)\leq e^c(v,t:\phi)
\]
for all $v\in W^c$, $\phi\in\mathcal{L}_J$, $t\in Jt$.
\end{lemma}
\begin{proof}
By the Truth lemma \ref{lem:gjlcscanonmodtruth} and the definition of $\mathfrak{M}^c(\mathcal{GJL}_{CS})$, we have $\mathcal{E}^c(v,t,\phi)=v(\phi_t)=e^c(v,t:\phi)$.
\end{proof}
Note, that for a $\mathsf{GJ}$-model $\mathfrak{M}=\langle W,R,\mathcal{E},e\rangle$, the strong evidence property $\mathcal{E}(w,t,\phi)\leq e(w,t:\phi)$ is equivalent with the stronger assertion $\mathcal{E}(w,t,\phi)=e(w,t:\phi)$ as $e(w,t:\phi)=\mathcal{E}(w,t,\phi)\odot e(w,\Box\phi)\leq\mathcal{E}(w,t,\phi)$ follows anyway by properties of $\odot$.
\begin{lemma}\label{lem:canonmodgjlcswelldef}
$\mathfrak{M}^c(\mathcal{GJL}_{CS})$ is a well-defined $\mathsf{GJL_{CS}}$-model.
\end{lemma}
\begin{proof}
Let $\mathfrak{M}^c(\mathcal{GJL}_{CS})=\langle W^c,R^c,\mathcal{E}^c,e^c\rangle$. We divide between the different possibilities for $\mathsf{GJL_{CS}}$:
\begin{description}
\item [$\mathsf{GJ_{CS}}$] We just need to check the three basic conditions on $\mathcal{E}^c$. Let $v\in W^c$ be arbitrary:
\begin{enumerate}[(i)]
\item As $v((Th_{\mathcal{GJ}_{CS}})^\star)=1$, we have $v((\phi\rightarrow\psi)_t\rightarrow (\phi_s\rightarrow\psi_{[t\cdot s]}))=1$, i.e. $v((\phi\rightarrow\psi)_t)\leq v(\phi_s)\Rightarrow v(\psi_{[t\cdot s]})$ and thus\\ $v((\phi\rightarrow\psi)_t)\odot v(\phi_s)\leq v(\psi_{[t\cdot s]})$. Thus, we have
\[
\mathcal{E}^c(v,t,\phi\rightarrow\psi)\odot\mathcal{E}^c(v,s,\phi)\leq\mathcal{E}^c(v,t\cdot s,\psi)
\]
by definition of $\mathfrak{M}^c$.
\item Again as $v((Th_{\mathcal{GJ}_{CS}})^\star)=1$, we have $v(\phi_t\rightarrow\phi_{[t+s]})=1$ and\\ $v(\phi_s\rightarrow\phi_{[t+s]})=1$, i.e. $v(\phi_t)\leq v(\phi_{[t+s]})$ and $v(\phi_s)\leq v(\phi_{[t+s]})$. Thus immediately $\mathcal{E}^c(v,t,\phi),\mathcal{E}^c(v,s,\phi)\leq\mathcal{E}^c(v,t+s,\phi)$, i.e.
\[
\mathcal{E}^c(v,t,\phi)\oplus\mathcal{E}^c(v,s,\phi)\leq\mathcal{E}^c(v,t+s,\phi).
\]
\item Let $c:\phi\in CS$. Then $c:\phi\in Th_{\mathcal{GJ}_{CS}}$ by ($CS$) in $\mathcal{GJ}_{CS}$, i.e. $\phi_c\in (Th_{\mathcal{GJ}_{CS}})^\star$ and thus for any $v\in W^c$, we have $v(\phi_c)=1$, i.e. $\mathcal{E}^c(v,c,\phi)=1$ for all $v\in W^c$.
\end{enumerate}
\item [$\mathsf{GJT_{CS}}$] We have $R^c(v,v)=1$ iff $\forall\phi\in\mathcal{L}_J:\forall t\in Jt: v(\phi_t)\leq v(\phi^\star)$ which follows, as $v((Th_{\mathcal{GJT}_{CS}})^\star)=1$, i.e. we have that $v(\phi_t\rightarrow\phi^\star)=1$ by axiom ($F$), i.e. $v(\phi_t)\Rightarrow v(\phi^\star)=1$, thus $v(\phi_t)\leq v(\phi^\star)$ for all $t\in Jt$ and any $\phi\in\mathcal{L}_J$. The rest follows as in the $\mathsf{GJ_{CS}}$-case.
\item [$\mathsf{GJ4_{CS}}$] We check the three additional conditions of $\mathsf{GJ4}$-models from Def. \ref{def:j4lpmodels}. The rest follows similarly to the $\mathsf{GJ_{CS}}$-case.
\begin{enumerate}[(i)]
\item Let $v,w\in W^c$. If $w(\phi_t)=\mathcal{E}^c(w,t,\phi)\leq\mathcal{E}^c(v,t,\phi)=v(\phi_t)$, the inequality is immediately satisfied. Thus, suppose that $w(\phi_t)>v(\phi_t)$. But as $w\in W^c$, we have $w(\phi_t)\leq w((t:\phi)_{!t})$, i.e. 
\[
v(\phi_t)<w(\phi_t)\leq w((t:\phi)_{!t}),
\]
i.e. $v(\phi_t)<w((t:\phi)_{!t})$. Thus 
\[
\exists\psi\in\mathcal{L}_J, s\in Jt:v(\psi^\star)< w(\psi_s).
\]
Thus, we have $R^c(w,v)=0$ in this case. 
\item Let $w,v,u\in W^c$. As $R^c$ is crisp, we have that either\\ $R^c(w,v)\odot R^c(v,u)=0$ or $R^c(w,v)\odot R^c(v,u)=1$. For the former, the inequality is trivially satisfied. Thus suppose $R^c(w,v)\odot R^c(v,u)=1$, i.e. $R^c(w,v)=R^c(v,u)=1$ and thus
\[
\forall\phi\in\mathcal{L}_J:\forall t\in Jt:w(\phi_t)\leq v(\phi^\star)
\]
and
\[
\forall\phi\in\mathcal{L}_J:\forall t\in Jt:v(\phi_t)\leq u(\phi^\star).
\]
Let $\phi\in\mathcal{L}_J$ and $t\in Jt$ be arbitrary. Then, by monotonicity, i.e. item (i), we have that $\mathcal{E}^c(w,t,\phi)\odot R^c(w,v)\leq\mathcal{E}^c(v,t,\phi)$, i.e. as $R^c(w,v)=1$ we have $w(\phi_t)=\mathcal{E}^c(w,t,\phi)\leq\mathcal{E}^c(v,t,\phi)=v(\phi_t)$. Thus, we have $w(\phi_t)\leq v(\phi_t)\leq u(\phi^\star)$. The claim follows as $\phi$ and $t$ were arbitrary. 
\item For any $w\in W^c$, as $w((Th_{\mathcal{GJ}4_{CS}})^\star)=1$, we have $w(\phi_t\rightarrow (t:\phi)_{!t})=1$ for any $\phi\in\mathcal{L}_J,t\in Jt$, i.e.
\[
\mathcal{E}^c(w,t,\phi)=w(\phi_t)\leq w((t:\phi)_{!t})=\mathcal{E}^c(w,!t,t:\phi)
\]
for any $\phi\in\mathcal{L}_J,t\in Jt$.
\end{enumerate}
\item [$\mathsf{GLP_{CS}}$] For $\mathcal{GLP}_{CS}$, as shown in the $\mathsf{GJT_{CS}}$-case, the presence of the axiom ($F$) makes the accessibility function reflexive. The other properties of a $\mathsf{GLP_{CS}}$-model follow from the $\mathsf{GJ4_{CS}}$-case.
\item [$\mathsf{GJ45_{CS}}$] As of Lem. \ref{lem:canonmodstrongevid}, the canonical model has the strong evidence property in any case. We thus just check the negative introspection property: Let $w\in W^c$, i.e. $w(\neg\phi_t\rightarrow (\neg t:\phi)_{?t})=1$, i.e. $\sim w(\phi_t)\leq w((\neg t:\phi)_{?t})$ and thus $\sim\mathcal{E}^c(w,t,\phi)\leq\mathcal{E}^c(w,?t,\neg t:\phi)$ for any $\phi\in\mathcal{L}_J,t\in Jt$. The other properties follow from the $\mathsf{GJ4_{CS}}$-case.
\item [$\mathsf{GJT45_{CS}}$] Again, through the $\mathsf{GJT_{CS}}$-case, the presence of axiom ($F$) makes the accessibility function reflexive. The other properties of a $\mathsf{GJT45_{CS}}$-model thus follows from the $\mathsf{GJ45_{CS}}$-case.
\end{description}
\end{proof}
Now we are ready to prove the main theorem of the paper, establishing strong completeness for all the introduced model classes and proof systems.
\begin{theorem}[Completeness of $\mathcal{GJL}_{CS}$]\label{thm:gjlcscompleteness}
For any $\Gamma\cup\{\phi\}\subseteq\mathcal{L}_J$, the following are equivalent:
\begin{enumerate}[(i)]
\item $\Gamma\vdash_{\mathcal{GJL}_{CS}}\phi$,
\item $\Gamma\models_{\mathsf{GJL_{CS}}\leq}\phi$,
\item $\Gamma\models_\mathsf{GJL_{CS}}\phi$,
\item $\Gamma\models_\mathsf{GJL_{CS}c}\phi$.
\end{enumerate}
\end{theorem}
\begin{proof}
We have (i) $\Rightarrow$ (ii) for each case $\mathcal{GJ}_{CS}$, $\mathcal{GJT}_{CS}$, $\mathcal{GJ}4_{CS}$, $\mathcal{GLP}_{CS}$, $\mathcal{GJ}45_{CS}$, $\mathcal{GJT}45_{CS}$ by Lem. \ref{lem:gjcssoundness}, Lem. \ref{lem:gjtcssoundness}, Lem. \ref{lem:j4lpcssoundness}, Lem. \ref{lem:45soundness} respectively. (ii) $\Rightarrow$ (iii) follows by Lem. \ref{lem:consequenceclassimpl} in any case, and as $\mathsf{GJL_{CS}c}$ is a subclass of $\mathsf{GJL_{CS}}$ for every case, we also have (iii) $\Rightarrow$ (iv). Thus, it suffices to prove (iv) $\Rightarrow$ (i).\\

For this, assume that $\Gamma\not\vdash_{\mathcal{GJL}_{CS}}\phi$. By Lem. \ref{lem:modalproptransinv}, this is equivalent with $\Gamma^\star\cup(Th_{\mathcal{GJL}_{CS}})^\star\not\vdash_\mathcal{G}\phi^\star$. By strong standard completeness of $\mathcal{G}$, there is an evaluation $v:\mathcal{L}_0^\star\to [0,1]$ such that $v(\Gamma^\star\cup(Th_{\mathcal{GJL}_{CS}})^\star)=1$ but $v(\phi^\star)<1$. By the former, we have $v((Th_{\mathcal{GJL}_{CS}})^\star)=1$ and thus $v\in W^c$. By the Truth Lemma \ref{lem:gjlcscanonmodtruth} for $\mathfrak{M}^c(\mathcal{GJL}_{CS})$, we thus have, by $v(\Gamma^\star)=1$, that $e^c(v,\Gamma)=1$ and by $v(\phi^\star)<1$, we have $e^c(v,\phi)<1$. By Lem. \ref{lem:canonmodgjlcswelldef}, $\mathfrak{M}^c(\mathcal{GJL}_{CS})$ is a well-defined accessibility crisp $\mathsf{GJL_{CS}}$-model such that $(\mathfrak{M}^c(\mathcal{GJL}_{CS}),v)\models\Gamma$ but $(\mathfrak{M}^c(\mathcal{GJL}_{CS}),v)\not\models\phi$ for $v\in\mathcal{D}(\mathfrak{M}^c(\mathcal{GJL}_{CS}))$. Thus $\Gamma\not\models_\mathsf{GJL_{CS}c}\phi$.
\end{proof}
We thus find that an analogue of the symmetry property for the accessibility function is not required to establish completeness of $\mathcal{GJ}45_{CS}$ and $\mathcal{GJT}45_{CS}$ w.r.t to their intended semantics, similarly to the classical boolean case.\\

We may also derive various corollaries from the strong completeness theorem.
\begin{definition}
A set $\Gamma\subseteq\mathcal{L}_J$ is called \emph{consistent in} $\mathcal{GJL}_{CS}$, if\\ $\Gamma\not\vdash_{\mathcal{GJL}_{CS}}\bot$.
\end{definition}
\begin{corollary}[Model existence]
Let $\Gamma\subseteq\mathcal{L}_J$. If $\Gamma$ is consistent w.r.t. $\mathcal{GJL}_{CS}$, then $\exists\mathfrak{M}\in\mathsf{GJL_{CS}}, w\in\mathcal{D}(\mathfrak{M}):(\mathfrak{M},w)\models\Gamma$.
\end{corollary}
\begin{proof}
Suppose $\Gamma$ is consistent in $\mathcal{GJL}_{CS}$, i.e. $\Gamma\not\vdash_{\mathcal{GJL}_{CS}}\bot$ and thus by Thm. \ref{thm:gjlcscompleteness}, $\Gamma\not\models_\mathsf{GJL_{CS}}\bot$, i.e. $\exists\mathfrak{M}\in\mathsf{GJL_{CS}},w\in\mathcal{D}(\mathfrak{M}):(\mathfrak{M},w)\models\Gamma$ directly per definition of $1$-entailment.
\end{proof}
We may also utilize the completeness theorem to show a strong form of conservativity for various G\"odel justification logics, over $\mathcal{G}$.
\begin{corollary}[Conservativity]
Let $\mathcal{GJL}_0\in\{\mathcal{GJ}_0$, $\mathcal{GJT}_0$, $\mathcal{GJ}4_0$, $\mathcal{GLP}_0\}$ and $CS$ be  a constant specification for $\mathcal{GJL}_0$. For any $\Gamma\cup\{\phi\}\subseteq\mathcal{L}_0$, if $\Gamma\vdash_{\mathcal{GJL}_{CS}}\phi$, then $\Gamma\vdash_\mathcal{G}\phi$.
\end{corollary}
\begin{proof}
Suppose $\Gamma\not\vdash_\mathcal{G}\phi$. By strong standard completeness of $\mathcal{G}$, Thm. \ref{thm:gssc}, we have $\Gamma\not\models\phi$, i.e. $\exists \hat{e}\in\mathsf{Ev}(\mathcal{L}_0):\hat{e}(\psi)=1$ for all $\psi\in\Gamma$ but $\hat{e}(\phi)<1$. We now construct a particular $\mathsf{GJL_{CS}}$-model, which encodes this faulty evaluation:\\

Let $\mathfrak{M}=\langle W,R,\mathcal{E},e\rangle$ be defined over
\begin{itemize}
\item $W:=\{w\}$,
\item $R(w,w):=1$,
\item $\mathcal{E}(w,t,\alpha):=1$ for all $t\in Jt$, $\alpha\in\mathcal{L}_J$,
\item $e(w,p):=\hat{e}(p)$ for all $p\in Var$.
\end{itemize}
As $\mathcal{E}(w,t,\alpha)=1$ for any choice of $t$ and $\alpha$, it clearly respects $CS$. Also, as all such entries of the evidence function are equal, we have
\begin{align*}
\mathcal{E}(w,t,\phi\rightarrow\psi)\odot\mathcal{E}(w,s,\phi)&\leq\mathcal{E}(w,t\cdot s,\psi),\\
\mathcal{E}(w,t,\phi)\oplus\mathcal{E}(w,s,\phi)&\leq\mathcal{E}(w,t+s,\phi),\\
\mathcal{E}(w,t,\phi)&\leq\mathcal{E}(w,!t,t:\phi).
\end{align*}
$R$ is trivially reflexive and (min-)transitive. As $W$ is a singleton, we have monotonicity directly as well. Now, we can prove:
\[
\text{For any }\alpha\in\mathcal{L}_0: e(w,\alpha)=\hat{e}(\alpha).
\]
For this, we proceed by induction on the structure of $\alpha$. As an induction base, for $\alpha=p\in Var$, we have $e(w,p)=\hat{e}(p)$ and also $e(w,\bot)=0=\hat{e}(\bot)$ per definition. The induction step for $\land$ and $\rightarrow$ follows from a straightforward application of the induction hypothesis.\\

Now, with $\mathfrak{M}$, we have found (in each case) a $\mathsf{GJL_{CS}}$-model such that $e(w,\psi)=\hat{e}(\psi)=1$ for all $\psi\in\Gamma$ as $\Gamma\subseteq\mathcal{L}_0$ but $e(w,\phi)=\hat{e}(\phi)<1$. Thus, per definition, $\Gamma\not\models_\mathsf{GJL_{CS}}\phi$ and thus by Thm. \ref{thm:gjlcscompleteness}, we have $\Gamma\not\vdash_{\mathcal{GJL}_{CS}}\phi$.
\end{proof}
\section{An alternative semantics over fuzzy Mkrtychev models}
Introduced in \cite{Mkr1997}, Mkrtychev models preceded Kripke-Fitting semantics for justification logics. From their perspective, Mkrtychev models essentially encode the necessary information concerning the justification modalities only via the admissible evidence function. In the following, we present G\"odel-Mkrtychev models for our various G\"odel justification logics for which we prove another strong completeness theorem.
\begin{definition}
A \emph{G\"odel-Mkrtychev model} is a structure $\mathfrak{M}=\langle\mathcal{E},e\rangle$ with
\begin{enumerate}
\item $\mathcal{E}:Jt\times\mathcal{L}_J\to [0,1]$,
\item $e:Var\to [0,1]$,
\end{enumerate}
where we have the following conditions on the corresponding admissible evidence function $\mathcal{E}$:
\begin{enumerate}[(i)]
\item $\mathcal{E}(t,\phi\rightarrow\psi)\odot\mathcal{E}(s,\phi)\leq\mathcal{E}(t\cdot s,\psi)$,
\item $\mathcal{E}(t,\phi)\oplus\mathcal{E}(s,\phi)\leq\mathcal{E}(t+s,\phi)$,
\end{enumerate}
for all $t,s\in Jt$ and $\phi,\psi\in\mathcal{L}_J$.
\end{definition}
In a similar spirit as before, $e$ extends to $\mathcal{L}_J$ via the following recursive rules:
\begin{itemize}
\item $e(\bot)=0$,
\item $e(\phi\land\psi)=e(\phi)\odot e(\psi)$,
\item $e(\phi\rightarrow\psi)=e(\phi)\Rightarrow e(\psi)$,
\item $e(t:\phi)=\mathcal{E}(t,\phi)$.
\end{itemize}
A G\"odel-Mkrtychev model respects a constant specification $CS$ if
\[
\mathcal{E}(c,\phi)=1\text{ for all }c:\phi\in CS.
\]
We denote the class of all G\"odel-Mkrtychev models by $\mathsf{GM}$ and for a class of $\mathsf{GM}$-models $\mathsf{C}$, we denote its subclass of models respecting a constant specification $CS$ by $\mathsf{C_{CS}}$.
\begin{definition}
Let $\mathfrak{M}=\langle\mathcal{E},e\rangle$ be a $\mathsf{GM}$-model and $\Gamma\cup\{\phi\}\subseteq\mathcal{L}_J$. We say
\begin{enumerate}[(i)]
\item $\phi$ \emph{is valid in} $\mathfrak{M}$, written $\mathfrak{M}\models\phi$, iff $e(\phi)=1$,
\item $\Gamma$ \emph{is valid in} $\mathfrak{M}$, written $\mathfrak{M}\models\Gamma$, iff $\forall\psi\in\Gamma:\mathfrak{M}\models\psi$.
\end{enumerate}
For $\mathsf{C}$ a class of $\mathsf{GM}$-models, we say
\begin{enumerate}[(i)]
\setcounter{enumi}{2}
\item $\phi$ \emph{is a consequence of} $\Gamma$ \emph{in} $\mathsf{C}$, written $\Gamma\models_{\mathsf{C}\leq}^M\phi$, iff $\forall\mathfrak{M}\in\mathsf{C}:e(\Gamma):=\inf_{\psi\in\Gamma}\{e(\psi)\}\leq e(\phi)$,
\item $\phi$ \emph{is a $1$-consequence of} $\Gamma$ \emph{in} $\mathsf{C}$, written $\Gamma\models_\mathsf{C}^M\phi$, iff $\forall\mathfrak{M}\in\mathsf{C}:\mathfrak{M}\models\Gamma$ implies $\mathfrak{M}\models\phi$.
\end{enumerate}
\end{definition}
A formula $\phi$ is called $\mathsf{C}$-\emph{valid}, for a class of $\mathsf{GM}$-models $\mathsf{C}$, if $\varnothing\models^M_\mathsf{C}\phi$. In this case, we also just write $\models^M_\mathsf{C}\phi$ similarly as before.
\begin{definition}
We call a G\"odel-Mkrtychev model $\mathfrak{M}=\langle\mathcal{E},e\rangle$ satisfying
\begin{enumerate}
\item $\mathcal{E}(t,\phi)\leq e(\phi)$ for all $t\in Jt,\phi\in\mathcal{L}_J$ a $\mathsf{GMT}$-model,
\item $\mathcal{E}(t,\phi)\leq\mathcal{E}(!t,t:\phi)$ for all $t\in Jt,\phi\in\mathcal{L}_J$ a $\mathsf{GM4}$-model,
\item (1) and (2) a $\mathsf{GMLP}$-model,
\item (2) and $\sim\mathcal{E}(t,\phi)\leq\mathcal{E}(?t,\neg t:\phi)$ for all $t\in Jt,\phi\in\mathcal{L}_J$ a $\mathsf{GM45}$-model,
\item (1) and (4) a $\mathsf{GMT45}$-model.
\end{enumerate}
\end{definition}
Again, in the following, let 
\[
\mathcal{GJL}_0\in\{\mathcal{GJ}_0,\mathcal{GJT}_0,\mathcal{GJ}4_0,\mathcal{GLP}_0,\mathcal{GJ}45_0,\mathcal{GJT}45_0\}
\]
and let $CS$ be a constant specification for $\mathcal{GJL}_0$. Let $\mathsf{GMJL_{CS}}$ represent the associated class of G\"odel-Mkrtychev models respecting that given constant specification $CS$.
\begin{lemma}\label{lem:mkrmodcsvalid}
Every formula that is deduced by the rule ($CS$) is valid in the class of $\mathsf{GMJL_{CS}}$-models.
\end{lemma}
\begin{proof}
Let $\mathfrak{M}=\langle\mathcal{E},e\rangle$ be a $\mathsf{GMJL_{CS}}$-model and let $c:\phi\in CS$. Then, as $\mathfrak{M}$ respects $CS$, we have $\mathcal{E}(c,\phi)=1$, i.e. $e(c:\phi)=1$ per definition for the extended $e$.
\end{proof}
\begin{lemma}[Soundness for $\mathsf{GMJL}$-models]\label{lem:soundnessmkrt}
For any $\Gamma\cup\{\phi\}\subseteq\mathcal{L}_J$:\\ $\Gamma\vdash_{\mathcal{GJL}_{CS}}\phi$ implies $\Gamma\models_{\mathsf{GMJL_{CS}}\leq}^M\phi$.
\end{lemma}
\begin{proof}
We divide between the different cases for $\mathcal{GJL}_{CS}$. Also, we just check the validity of the modal axioms in their respective classes. The rest follows from Lem. \ref{lem:mkrmodcsvalid} as before.
\begin{description}
\item [$\mathcal{GJ}_{CS}$] To see that ($J$) is valid, observe that 
\begin{align*}
e(t:(\phi\rightarrow\psi))\odot e(s:\phi)&=\mathcal{E}(t,\phi\rightarrow\psi)\odot\mathcal{E}(s,\phi)\\
                                         &\leq\mathcal{E}(t\cdot s,\psi)=e([t\cdot s]:\psi).
\end{align*}
Rearrangement follows again by properties of the residuum. To see that $(+)$ is valid, note that 
\begin{align*}
e(t:\phi)&=\mathcal{E}(t,\phi)\\
         &\leq\mathcal{E}(t+s,\phi)=e([t+s]:\phi),
\end{align*}
and similarly for the other version.
\item [$\mathcal{GJT}_{CS}$] Naturally, we have $e(t:\phi)=\mathcal{E}(t,\phi)\leq e(\phi)$, i.e. $e(t:\phi\rightarrow\phi)=1$ by the conditions on $\mathsf{GMT}$-models. The rest follows from the $\mathcal{GJ}_{CS}$-case.
\item [$\mathcal{GJ}4_{CS}$] We have that $e(t:\phi)=\mathcal{E}(t,\phi)\leq\mathcal{E}(!t,t:\phi)=e(!t:t:\phi)$ by the condition of $\mathsf{GM4}$-models. The rest follows again from the $\mathcal{GJ}_{CS}$-case.
\item [$\mathcal{GLP}_{CS}$] This case follows entirely from the $\mathcal{GJT}_{CS}$ and $\mathcal{GJ}4_{CS}$ cases.
\item [$\mathcal{GJ}45_{CS}$] We have
\begin{align*}
e(\neg t:\phi)&=\sim e(t:\phi)\\
              &=\sim\mathcal{E}(t,\phi)\\
              &\leq\mathcal{E}(?t,\neg t:\phi)=e(?t:\neg t:\phi),
\end{align*}
i.e. $e(\neg t:\phi\rightarrow ?t:\neg t:\phi)=1$. The rest follows from the $\mathcal{GJ}4_{CS}$-case.
\item [$\mathcal{GJT}45_{CS}$] Again, the cases for $\mathcal{GJ}45_{CS}$ and $\mathcal{GJT}_{CS}$ directly imply this one.
\end{description}
\end{proof}
\begin{definition}
Let $v\in\mathsf{Ev}(\mathcal{L}_0^\star)$ be such that $v((Th_{\mathcal{GJL}_{CS}})^\star)=1$. We define the \emph{canonical G\"odel-Mkrtychev model of} $\mathcal{GJL}_{CS}$ \emph{w.r.t.} $v$,
\[
\mathfrak{M}^c_v(\mathcal{GJL}_{CS})=\langle\mathcal{E}^c,e^c\rangle,
\]
over
\begin{enumerate}
\item $\mathcal{E}^c(t,\phi):=v(\phi_t)$ for all $\phi\in\mathcal{L}_J$, $t\in Jt$,
\item $e^c(p):=v(p)$ for all $p\in Var$.
\end{enumerate}
\end{definition}
\begin{lemma}\label{lem:mkrtcanontruth}
Let $\mathfrak{M}^c_v(\mathcal{GJL}_{CS})=\langle\mathcal{E}^c,e^c\rangle$ be the canonical G\"odel-Mkrtychev model of $\mathcal{GJL}_{CS}$ w.r.t to $v$. For all $\phi\in\mathcal{L}_J$: $e^c(\phi)=v(\phi^\star)$.
\end{lemma}
\begin{proof}
Induction on the structure of $\phi$:
\begin{description}
\item [(IB)] Let $\phi=p\in Var$, then $e^c(p)=v(p)=v(p^\star)$ per definition. If $\phi=\bot$, then $e^c(\bot)=0=v(\bot)=v(\bot^\star)$ per definition for the extension of an evaluation function.
\item [(IS)] We again divide between the different connectives of $\mathcal{L}_J$:

For $\phi=\psi\rightarrow\chi$ and $\phi=\psi\land\chi$, the claim follows again directly from \textbf{(IH)}, as we have e.g. $e^c(\psi\land\chi)=e^c(\psi)\odot e^c(\chi)=v(\psi^\star)\odot v(\chi^\star)=v(\psi^\star\land\chi^\star)=v((\psi\land\chi)^\star)$ and similarly for $\rightarrow$.

In comparison to Lem. \ref{lem:gjlcscanonmodtruth}, the claim for $\phi=t:\psi$ is even more straightforward, as we just have $e^c(t:\psi)=\mathcal{E}^c(t,\psi)=v(\psi_t)=v((t:\psi)^\star)$ per definition.
\end{description}
\end{proof}
\begin{lemma}\label{lem:mkrtcanonwelldef}
$\mathfrak{M}^c_v(\mathcal{GJL}_{CS})$ is a well-defined $\mathsf{GMJL_{CS}}$-model for any choice of $v\in\mathsf{Ev}(\mathcal{L}_0^\star)$ such that $v((Th_{\mathcal{GJL}_{CS}})^\star)=1$.
\end{lemma}
We omit the proof as it is similar to the proof of Lem. \ref{lem:canonmodgjlcswelldef}.
\begin{theorem}\label{thm:mkrtcompletenes}
For any $\Gamma\cup\{\phi\}\subseteq\mathcal{L}_J$, the following are equivalent:
\begin{enumerate}[(i)]
\item $\Gamma\vdash_{\mathcal{GJL}_{CS}}\phi$,
\item $\Gamma\models^M_{\mathsf{GMJL_{CS}}\leq}\phi$,
\item $\Gamma\models^M_\mathsf{GMJL_{CS}}\phi$.
\end{enumerate}
\end{theorem}
\begin{proof}
(i) $\Rightarrow$ (ii) follows from Lem. \ref{lem:soundnessmkrt} and (ii) $\Rightarrow$ (iii) follows naturally as before. Thus, we show (iii) $\Rightarrow$ (i). For this, suppose $\Gamma\not\vdash_{\mathcal{GJL}_{CS}}\phi$. Thus, by Lem. \ref{lem:modalproptransinv}, we have $\Gamma^\star\cup (Th_{\mathcal{GJL}_{CS}})^\star\not\vdash_{\mathcal{G}}\phi^\star$. By strong standard completeness of $\mathcal{G}$, we have that $\exists v\in\mathsf{Ev}(\mathcal{L}_0^\star):v(\Gamma^\star\cup (Th_{\mathcal{GJL}_{CS}})^\star)=1$ and $v(\phi^\star)<1$. Now, considering $\mathfrak{M}^c_v(\mathcal{GJL}_{CS})=\langle\mathcal{E}^c,e^c\rangle$, we have by the Truth Lemma \ref{lem:mkrtcanontruth}, that $e^c(\Gamma)=1$ and $e^c(\phi)<1$. By Lem. \ref{lem:mkrtcanonwelldef}, we have that $\mathfrak{M}^c_v(\mathcal{GJL}_{CS})$ is a well-defined $\mathsf{GMJL_{CS}}$-model. Thus $\Gamma\not\models_\mathsf{GMJL_{CS}}\phi$.
\end{proof}
\begin{corollary}[Model existence]
Let $\Gamma\subseteq\mathcal{L}_J$. If $\Gamma$ is consistent in $\mathcal{GJL}_{CS}$, then $\exists\mathfrak{M}\in\mathsf{GMJL_{CS}}:\mathfrak{M}\models\Gamma$.
\end{corollary}
\begin{proof}
Suppose $\Gamma\not\vdash_{\mathcal{GJL}_{CS}}\bot$, i.e. by Thm. \ref{thm:mkrtcompletenes} $\Gamma\not\models^M_\mathsf{GMJL_{CS}}\bot$, i.e. per definition of $1$-consequence in G\"odel-Mkrtychev, we have that $\exists\mathfrak{M}\in\mathsf{GMJL_{CS}}:\mathfrak{M}\models\Gamma$.
\end{proof}
We may derive a conservativity result for one of the remaining logics easier over the completeness theorem with respect to G\"odel-Mkrtychev models.
\begin{corollary}[Conservativity of $\mathcal{GJ}45_{CS}$]
Let $CS$ be a constant specification for $\mathcal{GJ}45_0$ and let $\Gamma\cup\{\phi\}\subseteq\mathcal{L}_0$. If we have $\Gamma\vdash_{\mathcal{GJ}45_{CS}}\phi$, then $\Gamma\vdash_\mathcal{G}\phi$.
\end{corollary}
\begin{proof}
Let $\Gamma\not\vdash_\mathcal{G}\phi$. By Thm. \ref{thm:gssc}, we have $\Gamma\not\models\phi$, i.e. $\exists\hat{e}\in\mathsf{Ev}(\mathcal{L}_0):\hat{e}(\psi)=1$ for all $\psi\in\Gamma$ with $\hat{e}(\phi)<1$. We consider the following $\mathsf{GM45_{CS}}$-model $\mathfrak{M}=\langle\mathcal{E},e\rangle$:
\begin{itemize}
\item $\mathcal{E}(t,\alpha)=1$ for all $t\in Jt$ and $\alpha\in\mathcal{L}_J$,
\item $e(p)=\hat{e}(p)$ for $p\in Var$.
\end{itemize}
$\mathcal{E}$ clearly respects $CS$ as before and it naively satisfies the sum and application laws for basic $\mathsf{GM}$-models. Also, $\mathcal{E}(t,\phi)\leq\mathcal{E}(!t,t:\phi)$ trivially follows and similarly direct, we have $\sim\mathcal{E}(t,\phi)=0\leq 1=\mathcal{E}(?t,\neg t:\phi)$. As before, we may prove $e(\phi)=\hat{e}(\phi)$ for any $\phi\in\mathcal{L}_0$ and thus we have found a $\mathsf{GM45_{CS}}$-model $\mathfrak{M}$ such that $\mathfrak{M}\models\Gamma$ and $\mathfrak{M}\not\models\phi$. Thus $\Gamma\not\models^M_\mathsf{GM45_{CS}}\phi$ and by the Completeness Theorem \ref{thm:mkrtcompletenes}, we thus have $\Gamma\not\vdash_{\mathcal{GJ}45_{CS}}\phi$.
\end{proof}
A construction of such a counter-model for the remaining logic $\mathcal{GJT}45_{CS}$ seems to be possible as well. However, a concrete initial advance proved itself to be rather complicated through the regularity condition $\mathcal{E}(t,\phi)\leq e(\phi)$ and we thus leave this as future work.\\

As in classical justification logic, we can find a way to identify G\"odel-Mkrtychev with single world G\"odel justification models.
\begin{definition}
Let $\mathfrak{M}=\langle\mathcal{E},e\rangle$ be a G\"odel-Mkrtychev model. 
\begin{enumerate}[(a)]
\item Its \emph{$0$-valued G\"odel justification counterpart model} $\overline{\mathfrak{M}}_0=\langle\{w\},\overline{R},\overline{\mathcal{E}},\overline{e}\rangle$ is defined with
\begin{enumerate}[1.]
\item $\overline{R}(w,w):=0$,
\item $\overline{\mathcal{E}}(w,t,\phi):=\mathcal{E}(t,\phi)$ for all $t\in Jt$, $\phi\in\mathcal{L}_J$,
\item $\overline{e}(w,p):=e(p)$ for all $p\in Var$.
\end{enumerate}
\item Its \emph{$1$-valued G\"odel justification counterpart model} $\overline{\mathfrak{M}}_1=\langle\{w\},\overline{R},\overline{\mathcal{E}},\overline{e}\rangle$ is defined similarly as in (a), where however $\overline{R}(w,w):=1$.
\end{enumerate}
\end{definition}
We find that, for a G\"odel-Mkrtychev model $\mathfrak{M}$, its $0$-valued counterpart model $\overline{\mathfrak{M}}_0$ really captures the content of the evaluation function of $\mathfrak{M}$.
\begin{lemma}\label{lem:zeromodadeq}
For any $\mathsf{GM}$-model $\mathfrak{M}=\langle\mathcal{E},e\rangle$ and its $0$-valued $\mathsf{GJ}$-counterpart model $\overline{\mathfrak{M}}_0=\langle\{w\},\overline{R},\overline{\mathcal{E}},\overline{e}\rangle$, we have $e(\phi)=\overline{e}(w,\phi)$ for any $\phi\in\mathcal{L}_J$.
\end{lemma}
\begin{proof}
We prove this by induction on the structure of $\phi$.
\begin{description}
\item [(IB)] Let $\phi=p\in Var$, then $\overline{e}(w,p)=e(p)$ per definition. Similarly, per definition, we have $\overline{e}(w,\bot)=0=e(\bot)$. 
\item [(IS)] We divide between the different connectives of $\mathcal{L}_J$:

For $\phi=\psi\rightarrow\chi$ and $\phi=\psi\land\chi$, the claim follows again directly from \textbf{(IH)}, as we have e.g. $\overline{e}(w,\psi\land\chi)=\overline{e}(w,\psi)\odot \overline{e}(w,\chi)=e(\psi)\odot e(\chi)=e(\psi\land\chi)$ and similarly for $\rightarrow$.

For $\phi=t:\psi$, we obtain
\begin{align*}
\overline{e}(w,t:\psi)&=(\overline{R}(w,w)\Rightarrow\overline{e}(w,\psi))\odot\overline{\mathcal{E}}(w,t,\psi)\\
                      &=(0\Rightarrow\overline{e}(w,\psi))\odot\overline{\mathcal{E}}(w,t,\psi)\\
                      &=\overline{\mathcal{E}}(w,t,\psi)=\mathcal{E}(t,\psi).
\end{align*}
\end{description}
\end{proof}
We also find, supposing a relatively weak condition on the class of models, that a $1$-valued counterpart of some G\"odel-Mkrtychev model has the same property.
\begin{lemma}\label{lem:onemodadeq}
Let $\mathfrak{M}=\langle\mathcal{E},e\rangle$ be a $\mathsf{GM}$-model, where $\mathcal{E}(t,\phi)\leq e(\phi)$ for all\\ $t\in Jt$, $\phi\in\mathcal{L}_J$. For its $1$-valued counterpart model $\overline{\mathfrak{M}}_1=\langle\{w\},\overline{R},\overline{\mathcal{E}},\overline{e}\rangle$, it holds that $e(\phi)=\overline{e}(w,\phi)$ for all $\phi\in\mathcal{L}_J$.
\end{lemma}
\begin{proof}
We prove this again by induction on the structure of $\phi$. For this, we may proceed as in the proof of Lem. \ref{lem:zeromodadeq} where only the case for $t:\psi$ changes:
\begin{align*}
\overline{e}(w,t:\psi)&=(\overline{R}(w,w)\Rightarrow\overline{e}(w,\psi))\odot\overline{\mathcal{E}}(w,t,\psi) &&\\
                      &=\overline{e}(w,\psi)\odot\overline{\mathcal{E}}(w,t,\psi)&&\text{(as }\overline{R}(w,w)=1)\\
                      &=e(\psi)\odot\mathcal{E}(t,\psi)&&\textbf{(IH)}\\
                      &=\mathcal{E}(t,\psi)=e(t:\psi)&&\text{(as }\mathcal{E}(t,\psi)\leq e(\psi)).
\end{align*}
\end{proof}
The following lemma now states that any such counterpart structure is actually a well-defined model. Even more so, we find that for a G\"odel-Mkrtychev model from one of the basic model classes introduced, either its $0$- or $1$-valued counterpart model is a member of the corresponding class of G\"odel justification models.
\begin{lemma}\label{lem:zeromodprop}
For any $\mathsf{GM}$, $\mathsf{GM4}$, $\mathsf{GM45}$-model $\mathfrak{M}=\langle\mathcal{E},e\rangle$, $\overline{\mathfrak{M}}_0$ is a well-defined $\mathsf{GJ}$, $\mathsf{GJ4}$, $\mathsf{GJ45}$-model, respectively.
\end{lemma}
\begin{proof}
Let $\overline{\mathfrak{M}}_0=\langle\{w\}$, $\overline{R},\overline{\mathcal{E}},\overline{e}\rangle$. We divide between the various cases:
\begin{description}
\item [$\mathsf{GM}$] For any $t\in Jt$, $\phi,\psi\in\mathcal{L}_J$ we obtain
\begin{align*}
\overline{\mathcal{E}}(w,t,\phi\rightarrow\psi)\odot\overline{\mathcal{E}}(w,s,\phi)&=\mathcal{E}(t,\phi\rightarrow\psi)\odot\mathcal{E}(s,\phi)\\
                                                                                    &\leq\mathcal{E}(t\cdot s,\psi)=\overline{\mathcal{E}}(w,t\cdot s,\psi),
\end{align*}
and
\begin{align*}
\overline{\mathcal{E}}(w,t,\phi)\oplus\overline{\mathcal{E}}(w,s,\phi) &=\mathcal{E}(t,\phi)\oplus\mathcal{E}(s,\phi)\\
                                                                       &\leq\mathcal{E}(t+s,\phi)=\overline{\mathcal{E}}(w,t+s,\phi).
\end{align*}
\item [$\mathsf{GM4}$] We have, for any $\phi\in\mathcal{L}_J$, $t\in Jt$ that
\[
\overline{\mathcal{E}}(w,t,\phi)\odot \overline{R}(w,w)=0\leq\overline{\mathcal{E}}(w,t,\phi),
\]
i.e. $\overline{\mathcal{E}}$ is monotone w.r.t. $\overline{R}$. Also, we trivially have that $\overline{R}$ is min-transitive. Lastly, we obtain
\[
\overline{\mathcal{E}}(w,t,\phi)=\mathcal{E}(t,\phi)\leq\mathcal{E}(!t,t:\phi)=\overline{\mathcal{E}}(w,!t,t:\phi).
\]
The rest follows as in the case for $\mathsf{GM}$.
\item [$\mathsf{GM45}$] For any $\phi\in\mathcal{L}_J$, $t\in Jt$, we have that
\[
\sim\overline{\mathcal{E}}(w,t,\phi)=\sim\mathcal{E}(t,\phi)\leq\mathcal{E}(?t,\neg t:\phi)=\overline{\mathcal{E}}(w,?t,\neg t:\phi)
\]
and that
\[
\overline{\mathcal{E}}(w,t,\phi)=\mathcal{E}(t,\phi)=e(t:\phi)=\overline{e}(w,t:\phi)
\]
where the last equality follows from Lem. \ref{lem:zeromodadeq}. The rest follows as in the case of $\mathsf{GM4}$.
\end{description}
\end{proof}
\begin{lemma}\label{lem:onemodprop}
For any $\mathsf{GMT}$, $\mathsf{GMLP}$, $\mathsf{GMT45}$-model $\mathfrak{M}=\langle\mathcal{E},e\rangle$, $\overline{\mathfrak{M}}_1$ is a well-defined $\mathsf{GJT}$, $\mathsf{GLP}$, $\mathsf{GJT45}$-model, respectively.
\end{lemma}
\begin{proof}
Let $\overline{\mathfrak{M}}_1=\langle\{w\}$, $\overline{R},\overline{\mathcal{E}},\overline{e}\rangle$. We again divide between the various cases:
\begin{description}
\item [$\mathsf{GMT}$] Reflexivity follows per definition and the other inequalities follow as in Lem. \ref{lem:zeromodprop} in the case for $\mathsf{GM}$.
\item [$\mathsf{GMLP}$] For any $\phi\in\mathcal{L}_J$, $t\in Jt$, we have that
\[
\overline{\mathcal{E}}(w,t,\phi)\odot \overline{R}(w,w)=\min\{\overline{\mathcal{E}}(w,t,\phi),1\}=\overline{\mathcal{E}}(w,t,\phi),
\]
confirming monotonicity of $\overline{\mathcal{E}}$ w.r.t. $\overline{R}$. $\overline{R}$ is again trivially min-transitive and the rest follows as in the case for $\mathsf{GMT}$ and for $\mathsf{GM4}$ in Lem. \ref{lem:zeromodprop}.
\item [$\mathsf{GMT45}$] Negative introspectivity of $\overline{\mathcal{E}}$ follows as in Lem. \ref{lem:zeromodprop} in the case of $\mathsf{GM45}$ and we obtain
\[
\overline{\mathcal{E}}(w,t,\phi)=\mathcal{E}(t,\phi)=e(t:\phi)=\overline{e}(w,t:\phi),
\]
this time by Lem. \ref{lem:onemodadeq}. The rest follows as in the case for $\mathsf{GMLP}$.
\end{description}
\end{proof}
As a consequence, we obtain that any logic introduced here has the simple finite model property w.r.t. G\"odel justification models. For this, again let $\mathcal{GJL}_0\in\{\mathcal{GJ}_0,\mathcal{GJT}_0,\mathcal{GJ}4_0,\mathcal{GLP}_0,\mathcal{GJ}45_0,\mathcal{GJT}45_0\}$, $CS$ be a constant specification for $\mathcal{GJL}_0$ and $\mathsf{GJL_{CS}}$ and $\mathsf{GMJL_{CS}}$ be the corresponding model classes of $\mathsf{GJ}$-models and $\mathsf{GM}$-models for which we have proved completeness, respectively.
\begin{theorem}
For any $\Gamma\cup\{\phi\}\subseteq\mathcal{L}_J$, if $\Gamma\not\vdash_{\mathcal{GJL}_{CS}}\phi$, then there is a simply finite $\mathsf{GJL_{CS}}$-model $\mathfrak{M}$ and $w\in\mathcal{D}(\mathfrak{M})$ with $(\mathfrak{M},w)\models\Gamma$ but $(\mathfrak{M},w)\not\models\phi$.
\end{theorem}
\begin{proof}
Suppose that $\Gamma\not\vdash_{\mathcal{GJL}_{CS}}\phi$. By Thm. \ref{thm:mkrtcompletenes}, we have $\Gamma\not\models^M_\mathsf{GMJL_{CS}}\nobreak\phi$, i.e. there is a $\mathsf{GMJL_{CS}}$-model $\mathfrak{N}$ such that $\mathfrak{N}\models\Gamma$ but $\mathfrak{N}\not\models\phi$. If
\[
\mathfrak{N}\in\{\mathsf{GM},\mathsf{GM4},\mathsf{GM45}\},
\] 
let $\mathfrak{M}=\overline{\mathfrak{N}}_0$. On the other hand, if 
\[
\mathfrak{N}\in\{\mathsf{GMT}, \mathsf{GMLP}, \mathsf{GMT45}\},
\] 
take $\mathfrak{M}=\overline{\mathfrak{N}}_1$. Let $\mathfrak{N}=\langle\mathcal{E},e\rangle$, then for the latter we have that $\mathcal{E}(t,\phi)\leq e(\phi)$ for all $\phi\in\mathcal{L}_J$, $t\in Jt$. We obtain that $\mathfrak{M}$ either way respects the constant specification $CS$ and that $(\mathfrak{M},w)\models\Gamma$ but $(\mathfrak{M},w)\not\models\phi$. This follows from Lem. \ref{lem:zeromodadeq} and Lem. \ref{lem:onemodadeq} respectively. Also, by Lem. \ref{lem:zeromodprop} or Lem. \ref{lem:onemodprop} respectively, we have that $\mathfrak{M}$ is either way a well-defined $\mathsf{GJL_{CS}}$-model. As $\mathcal{D}(\mathfrak{M})=\{w\}$, $\mathfrak{M}$ is simply finite.
\end{proof}
The above considerations show that the content of the possible-worlds part of G\"odel justification models may be \emph{completely encoded} into the admissible evidence function, similarly to the classical case. This shows in which strong ways justifications add information to epistemic scenarios also in the many-valued setting.

This is, at first, in contrast with standard G\"odel modal logic, where the fundamental proof systems do not enjoy the finite model property w.r.t. their fundamental semantics over G\"odel-Kripke models, which are G\"odel justification models without the admissible evidence function $\mathcal{E}$. However, the function $\mathcal{E}$ in the G\"odel-Fitting models is essentially an infinite object which is why this reduction to finite sets of possible worlds is in a sense really just a \emph{simple} and not a true finite model property.

In the literature for the classical case, there are various ways of reducing the admissible evidence function to a finite object, e.g. by considering so called \emph{evidence bases} (see \cite{Kuz2008}), and it remains as an open problem to further investigate the applicability of those to strengthen the above results.
\section{An application of vague justifications}
T-norm based fuzzy logic is well known for its capabilities of modeling and resolving argumentation scenarios involving vague propositions. As shown by Ghari's treatment in \cite{Gha2016} for an epistemic variant of the famous sorites-style (slippery-slope) paradox, fuzzy justification logics are very well adapted to treat vague justifications for (vague) propositions. 

A difference between G\"odel justification logics and other representatives of the class of fuzzy justification logics is that the G\"odel-case can also handle $[0,1]$-valued Fitting models with a \emph{fuzzy} accessibility function. This proves to be advantageous in modeling certain epistemic scenarios, e.g. the following, where we present a slippery slope argument taking place inside the accessibility function, which is determined by a vague predicate. Consider the following situation:
\begin{displayquote}
Imagine a person in a room. His room has a temperature of 25 $^{\circ}$C, which he considers warm and his feeling of the warmth provides evidence for the room being warm. As he is living in a country near the equator, he can only imagine the room warm, that is any possible different situation of his reality involves the room being warm. He thus has a justified true belief of the room being warm. He agrees, that if one would lower the temperature in his room, it would become gradually less warm, with 0 $^{\circ}$C being $0$ in degree of warmness. He also agrees that a temperature change of $\pm$1 $^{\circ}$C in a warm room will not make it cold. He thus considers a situation possible where the room temperature is 24 $^{\circ}$C and iterating this argument he considers a situation possible where the temperature is 0 $^{\circ}$C, which is however not warm anymore, contradicting his justified true belief.
\end{displayquote}
For a formalization, we use a propositional variable $w$ for the proposition \emph{the room is warm} and a justification variable $x$ for representing the feeling of the agent that the temperature of the actual room is warm.

We may first set up an accessibility-crisp G\"odel-Fitting model
\[
\mathfrak{M}=\langle\{\mathbf{25}, \mathbf{24}, \dots, \mathbf{0}\},R,\mathcal{E},e\rangle
\]
where the worlds correspond to various temperature scenarios of the room, encoded in their name. We may set $\mathcal{E}(\mathbf{25},x,w)=1$ as a natural modeling assumption for the agents feeling of the temperature.\footnote{In favor of simplicity, we ignore the other values of the evidence function (and similarly so for the forthcoming evaluation and accessibility functions).} It is natural, in the different room scenarios $\mathbf{25},\dots,\mathbf{0}$ considered as possible worlds, that the proposition $w$, representing the actual degree of warmness, is modeled to decline in truth value, e.g. by
\[
e(\mathbf{T},w)=\big (\frac{\mathbf{T}}{25}\big )^6
\]
for $\mathbf{T}\in\{\mathbf{25},\dots,\mathbf{0}\}$. For the accessibility function, it is at first reasonable to set $R(\mathbf{25},\mathbf{25})=1$. By further formalizing the assumption from the presented example regarding the accessibility function, and being restricted to the values $0,1$, we are required to set $R(\mathbf{25},\mathbf{24})=1$, as we want that if $R(\mathbf{25},\mathbf{T})>0$, then $R(\mathbf{25},\mathbf{T-1})>0$. Continuing this, we are thus necessarily left with an accessibility function $R$ being characterized by $R(\mathbf{25},\mathbf{T})=1$ for any $\mathbf{T}\in\{\mathbf{25},\dots,\mathbf{0}\}$. We may visualize this model as follows:
\begin{center}
\begin{tikzpicture}
\node[draw,circle] at (4,2) (T) {$\mathbf{25}$};
\node[draw,circle] at (0,0) (N) {$\mathbf{24}$};
\node[draw,circle] at (2,0) (E) {$\mathbf{23}$};
\node[draw,circle] at (4,0) (S) {$\mathbf{22}$};
\node at (6,0) (D) {$\dots$};
\node at (5.25,0.75) (D') {$\dots$};
\node[draw,circle] at (8,0) (O) {$\mathbf{0}$};
\draw (T) -- node[right] {$1$} (N);
\draw (T) -- node[right] {$1$} (E);
\draw (T) -- node[right] {$1$} (S);
\draw (T) -- node[above] {$1$} (O);
\end{tikzpicture}
\end{center}
This yields by the semantics of G\"odel-Fitting models:
\[
e(\mathbf{25},x:w)=\mathcal{E}(\mathbf{25},x,w)\odot e(\mathbf{25},\Box w)=e(\mathbf{0},w)=0.
\]
An accessibility-crisp model is thus not capable, given the premises, to resolve this argument. The problem here is essentially that, although the accessibility-crisp model is able to model the vagueness of the propositions and justifications properly, it is not able to model the vagueness determining the accessibility function.\\

In the case of true $[0,1]$-valued Fitting models however, we can formalize the much more natural assumption (which is still in accordance to the described situation) that the accessibility degrees decrease while still staying positive and even that they decrease much faster than the degree of warmness. We may thus require $R(\mathbf{25},\mathbf{T})>R(\mathbf{25},\mathbf{T-1})>0$. And the many-valuedness of the accessibility function allows a reasonable concrete implementation of this assumption by e.g. setting
\[
R'(\mathbf{25},\mathbf{T})=\big (\frac{\mathbf{T}}{25}\big )^7
\]
for any $\mathbf{T}\in\{\mathbf{25},\dots,\mathbf{0}\}$. This is also in line with the assumption made before that the current world $\mathbf{25}$ is totally accessible, as $R'(\mathbf{25},\mathbf{25})=1$. The resulting model $\mathfrak{M}'=\langle\{\mathbf{25}$, $\dots$, $\mathbf{0}\}$, $R',\mathcal{E}',e'\rangle$, with $\mathcal{E}'=\mathcal{E}$ and $e'=e$, can be visualized as follows:
\begin{center}
\begin{tikzpicture}
\node[draw,circle] at (4,2) (T) {$\mathbf{25}$};
\node[draw,circle] at (0,0) (N) {$\mathbf{24}$};
\node[draw,circle] at (2,0) (E) {$\mathbf{23}$};
\node[draw,circle] at (4,0) (S) {$\mathbf{22}$};
\node at (6,0) (D) {$\dots$};
\node at (5.25,0.75) (D') {$\dots$};
\node[draw,circle] at (8,0) (O) {$\mathbf{0}$};
\draw (T) -- node[right] {$\approx\frac{3}{4}$} (N);
\draw (T) -- node[right] {$\approx\frac{1}{2}$} (E);
\draw (T) -- node[right] {$\approx\frac{2}{5}$} (S);
\draw (T) -- node[above] {$0$} (O);
\end{tikzpicture}
\end{center}
We obtain, as desired: $e'(\mathbf{25},x:w)=\mathcal{E}'(\mathbf{25},x,w)\odot e'(\mathbf{25},\Box w)=1$ as we have
\[
R'(\mathbf{25},\mathbf{T})=\big (\frac{\mathbf{T}}{25}\big )^7\leq\big (\frac{\mathbf{T}}{25}\big )^6=e'(\mathbf{T},w)
\]
for any $\mathbf{T}\in\{\mathbf{25},\dots,\mathbf{0}\}$ and thus $e'(\mathbf{25},\Box w)=1$.
\section{Conclusions and further directions}
In this note, we exhibited fuzzy analogies to concepts from justification logic. More specifically, we replaced classical boolean propositional logic with G\"odel fuzzy logic as a base for the modal extensions of justification logic. With this, we translated the common semantical approach via Kripke-Fitting possible world semantics to a many-valued setting and, in contrast to previous approaches to fuzzy justification logic, we considered models with a fuzzy accessibility function. We then provided Hilbert-style axiomatic proof systems for the resulting analogous model classes of the most common representatives of classical justification logic, proved a strong completeness theorem for all of them and deduced various corollaries in the following. With G\"odel-Mkrtychev models, we also translated a second semantical access point to justification logic besides the Kripke-Fitting approach into the setting of G\"odel logic for which we provided a second strong completeness theorem for the here introduced proof systems, which is another similarity G\"odel justification logic bears with the classical version. A conversion of G\"odel-Mkrtychev models to G\"odel justification models is exhibited at the end.\\

However, this paper is only one of a few regarding the topic of fuzzy justification logics and there remain a lot of interesting questions yet still to be answered. In the following, we give pointers to some possible directions.
\subsection{Forgetful projection and realization}
Very prominent results in the classical case are the realization and projection theorems by Artemov, relating a justification logic to a classical modal counterpart in the sense that for every theorem of
\begin{enumerate}
\item the classical modal proof calculus, there exists an assignment of justification terms to the occurrences of the standard necessity modality $\Box$ such that the resulting formula is a theorem in the calculus of the justification logic, (Realization),
\item the justification proof calculus, replacing every justification modality by the standard necessity operator $\Box$ results in a theorem of the classical modal calculus, (Forgetful Projection).
\end{enumerate} 
For the systems $\mathcal{G}_\Box$, $\mathcal{G}_\Box+\mathbf{T}_\Box$, $\mathcal{G}_\Box+\mathbf{4}_\Box$ and $\mathcal{G}_\Box+\mathbf{T}_\Box+\mathbf{4}_\Box$ established in \cite{CR2010} and the systems $\mathcal{GJ}_{CS}$, $\mathcal{GJT}_{CS}$, $\mathcal{GJ}4_{CS}$ and $\mathcal{GLP}_{CS}$ respectively, introduced in this paper, the Forgetful Projection property follows immediately. It shall be very interesting to see as of if and how the Realization Theorem can be proved in the case of G\"odel justification logic and standard G\"odel modal logic.
\subsection{Adding the Baaz-Delta and truth constants}
One common extension to G\"odel (or in general fuzzy) logics is the Baaz-Delta operator $\Delta$, introduced by Baaz in \cite{Baa1996} as a unary crisp projection operator stipulated over the truth function $\delta:[0,1]\to[0,1]$ with
\[
\delta(x)=\begin{cases}1, &\text{if }x=1\\0, &\text{otherwise}\end{cases}.
\]
In plain G\"odel logic, this operator is not definable and thus adds expressive strength. Another common extension is the incorporation of countably many truth-value constants into the language, i.e. adding formulas of the form $\bar{c}$ for $c\in C\subseteq [0,1]$ for a countable $C$ and stipulating an evaluation of $\bar{c}$ in every case as its represented value $c$. These extensions, especially in combination with one another, are by now well-studied in the framework of basic propositional mathematical fuzzy logic, see e.g. \cite{EGGN2007}.\\

As an advantage in the case of G\"odel justification logic, besides gaining general expressive strength, it might also be interesting to consider \emph{graded justification assertions}, that is formulas of the type
\begin{align*}
t:_c\phi &:=\bar{c}\rightarrow t:\phi,\\
t:^c\phi &:=t:\phi\rightarrow\bar{c},\\
t\stackrel{c}{:}\phi &:= t:_c\phi\land t:^c\phi,
\end{align*}
for truth constants $\bar{c}$ with the intuitive reading of having \emph{at least}, \emph{at most} and \emph{exactly} a certainty degree of $c$ of regarding $t$ as a justification of $\phi$. These were already studied by Ghari in \cite{Gha2014}, \cite{Gha2016} in the context of justification logic over rational Pavelka logic and considered conceptually different before also by Milnikel in \cite{Mil2014}. The additional presence of the crisp projection $\Delta$ might even create various other possibilities for internal definitions of model-theoretically interesting justification assertions.
\subsection{Using other fuzzy logics as a base}
Among the other prominent representatives for systems of mathematical fuzzy logic, G\"odel logic is in general a well-behaved example (e.g. being the only instance enjoying the classical deduction theorem), as this paper additionally exhibits through the similarity of G\"odel justification logic to many classical cases. However, for future work it might be interesting to consider these other common examples as choices of bases for justification logic. Investigations in this already include Ghari's work \cite{Gha2016}, where he studies the case of using rational Pavelka logic, i.e. \L ukasiewicz logic with truth constants $\bar{r}$ for every $r\in [0,1]\cap\mathbb{Q}$. But also \L ukasiewicz logic alone as well as Product logic shall be very interesting to consider.\footnote{As said before, in \cite{Gha2014}, Ghari already exhibited the basics of some of these various other systems over crisp frames.}\\

As however already exhibited in e.g. \cite{VEG2017}, \cite{BEG2007}, \cite{Vid2015}, these logics prove themselves already to be quite untamed in the context of classical modal operators, as e.g. the modal axiom (K)
\[
\Box(\phi\rightarrow\psi)\rightarrow (\Box\phi\rightarrow\Box\psi)
\]
is no longer valid over the class of all corresponding Kripke models with \emph{fuzzy} accessibility function. It should be interesting to see how these logics cooperate with an extension in the spirit of justification logic, both in fuzzy-framed and crisp-framed models and if they are respectively axiomatizable.
\bibliographystyle{plain}
\bibliography{ref}

\begin{thebibliography}{10}

\bibitem{Art1995}
S.~Artemov.
\newblock {Operational Modal Logic}.
\newblock Technical Report MSI 95-29, Cornell University, 1995.
\newblock Ithaca, NY.

\bibitem{Art2001}
S.~Artemov.
\newblock {Explicit Provability and Constructive Semantics}.
\newblock {\em The Bulleting of Symbolic Logic}, 7(1):1--36, 2001.

\bibitem{Art2008}
S.~Artemov.
\newblock {The logic of justification}.
\newblock {\em The Review of Symbolic Logic}, 1(4):477--513, 2008.

\bibitem{Baa1996}
M.~Baaz.
\newblock {Infinite-valued G\"odel logics with 0-1-projections and
  relativizations}.
\newblock In {\em Proc. G\"odel'96, Logic, Foundations of Mathematics, Computer
  Science and Physics}, volume~6 of {\em Lecture Notes in Logic}, pages 23--33.
  Springer, 1996.

\bibitem{BPZ2007}
M.~Baaz, N.~Preining, and R.~Zach.
\newblock {First-Order G\"odel Logics}.
\newblock {\em Annals of Pure and Applied Logic}, 147(1--2):23--47, 2007.

\bibitem{BZ1998}
M.~Baaz and R.~Zach.
\newblock {Compact Propositional G\"odel logics}.
\newblock In {\em Proceedings of the 28th International Symposium on
  Multiple-Valued Logic}, pages 108--113. IEEE Computer Society Press, 1998.

\bibitem{BEG2007}
F.~Bou, F.~Esteva, and L.~Godo.
\newblock {Modal systems based on many-valued loigcs}.
\newblock In {\em New Dimensions in Fuzzy Logic and Related Technologies,
  Proceedings of the 5th EUSFLAT Conference}, volume~1, pages 177--182.
  Universitas Ostraviensis, 2007.

\bibitem{CR2009}
X.~Caicedo and R.~Rodriguez.
\newblock {A Godel Modal Logic}.
\newblock {\em ArXiv e-prints}, 2009.
\newblock arXiv, math.LO, 0903.2767.

\bibitem{CR2010}
X.~Caicedo and R.~Rodriguez.
\newblock {Standard G\"odel Modal Logics}.
\newblock {\em Studia Logica}, 94(2):189--214, 2010.

\bibitem{CR2015}
X.~Caicedo and R.~Rodriguez.
\newblock {Bi-modal G\"odel logic over $[0,1]$-valued Kripke frames}.
\newblock {\em Journal of Logic and Computation}, 25(1):37--55, 2015.

\bibitem{Dum1959}
M.~Dummett.
\newblock {A propositional calculus with denumerable matrix}.
\newblock {\em Journal of Symbolic Logic}, 24(2):97--106, 1959.

\bibitem{EGGN2007}
F.~Esteva, J.~Gispert, L.~Godo, and C.~Noguera.
\newblock {Adding truth-constants to logics of continuous t-norms:
  Axiomatization and completeness results}.
\newblock {\em Fuzzy Sets and Systems}, 158(6):597--618, 2007.

\bibitem{FL2015}
T.-F. Fan and C.-J. Liau.
\newblock {A Logic for Reasoning about Justified Uncertain Beliefs}.
\newblock In {\em Proceedings of the 24th International Joint Conference of
  Artificial Intelligence (IJCAI-15)}, pages 2948--2954. AAAI Press, 2015.

\bibitem{Fit1991}
M.~Fitting.
\newblock {Many-valued modal logics}.
\newblock {\em Fundamenta Informaticae}, 15(3--4):235--254, 1991.

\bibitem{Fit1992}
M.~Fitting.
\newblock {Many-valued modal logics II}.
\newblock {\em Fundamenta Informaticae}, 17(1--2):55--73, 1992.

\bibitem{Fit2003}
M.~Fitting.
\newblock {A Semantics for the Logic of Proofs}.
\newblock Technical Report TR-2003012, City University of New York, 2003.
\newblock PhD Program in Computer Science.

\bibitem{Fit2005}
M.~Fitting.
\newblock {The logic of proofs, semantically}.
\newblock {\em Annals of Pure and Applied Logic}, 132(1):1--25, 2005.

\bibitem{Gha2014}
M.~Ghari.
\newblock {Justification Logics in a Fuzzy Setting}.
\newblock {\em ArXiv e-prints}, 2014.
\newblock arXiv, math.LO, 1407.4647.

\bibitem{Gha2016}
M.~Ghari.
\newblock {Pavelka-style fuzzy justification logics}.
\newblock {\em Logic Journal of the IGPL}, 24(5):743--773, 2016.

\bibitem{Goe1932}
K.~G\"odel.
\newblock {Zum intuitionistischen Aussagenkalk\"ul}.
\newblock {\em Anzeiger der Akademie der Wissenschaften in Wien}, 69:65--66,
  1932.

\bibitem{Goe1933}
K.~G\"odel.
\newblock {Eine Interpretation des intuitionistischen Aussagenkalk\"uls}.
\newblock {\em Ergebnisse eines mathematischen Kolloquiums}, 4:39--40, 1933.

\bibitem{Goe1938}
K.~G\"odel.
\newblock {Vortrag bei Zilsel}.
\newblock In {\em Kurt G\"odel Collected Works}, volume III, pages 62--113.
  Oxford, 1938.
\newblock Transscripted Lecture.

\bibitem{Haj1998}
P.~H\'ajek.
\newblock {\em {Metamathematics of Fuzzy Logic}}, volume~4 of {\em Trends in
  Logic}.
\newblock Kluwer, Dordrecht, 1998.

\bibitem{Hin1962}
J.~Hintikka.
\newblock {\em {Knowledge and Belief}}.
\newblock Cornell University Press, Ithaca, NY, 1962.

\bibitem{Hor1969}
A.~Horn.
\newblock {Logic with Truth Values in a Linearly Ordered Heyting Algebra}.
\newblock {\em Journal of Symbolic Logic}, 34(3):395--408, 1969.

\bibitem{KMP2000}
E.~P. Klement, R.~Mesiar, and E.~Pap.
\newblock {\em {Triangular Norms}}, volume~8 of {\em Trends in Logic}.
\newblock Springer Netherlands, 2000.

\bibitem{KOS2016}
I.~Kokkinis, Z.~Ognjanovi\'c, and T.~Studer.
\newblock {Probabilistic Justification Logic}.
\newblock In {\em Proceedings of Logical Foundations of Computer Science
  LFCS'16}, volume 9537 of {\em Lecture Notes in Computer Science}, pages
  174--186. Springer, 2016.

\bibitem{Kuz2008}
R.~Kuznets.
\newblock {\em {Complexity Issues in Justification Logic}}.
\newblock PhD thesis, City University of New York Graduate Center, 2008.

\bibitem{Mil2014}
R.~Milnikel.
\newblock {The Logic of Uncertain Justifications}.
\newblock {\em Annals of Pure and Applied Logic}, 165(1):305--315, 2014.

\bibitem{Mkr1997}
A.~Mkrtychev.
\newblock {Models for the logic of proofs}.
\newblock In {\em Proceedings of Logical Foundations of Computer Science
  LFCS'97}, volume 1234 of {\em Lecture Notes in Computer Science}, pages
  266--275. Springer, 1997.

\bibitem{Vid2015}
A.~Vidal.
\newblock {\em {On modal expansions of t-norm based logics with rational
  constants}}.
\newblock PhD thesis, Artificial Intelligence Research Institute (IIIA - CSIC)
  and Universitat de Barcelona, 2015.

\bibitem{VEG2017}
A.~Vidal, F.~Esteva, and L.~Godo.
\newblock {On modal extensions of Product fuzzy logic}.
\newblock {\em Journal of Logic and Computation}, 27(1):299--336, 2017.

\bibitem{Zad1965}
L.~A. Zadeh.
\newblock {Fuzzy Sets}.
\newblock {\em Information and Control}, 8(3):338--353, 1965.

\end{thebibliography}
\end{document}